\documentclass[final,reqno]{amsart}

\usepackage{graphicx}
\usepackage{amsmath,amsthm,amsfonts,amscd,amssymb}
\usepackage[color,notref,notcite]{showkeys}
\usepackage{url}
\usepackage{subeqnarray}
\usepackage[color,notref,notcite]{showkeys} 
\definecolor{labelkey}{rgb}{0.6,0,1}
\usepackage{enumitem}
\usepackage{algorithm}
\usepackage{algpseudocode}
\usepackage{citesort}

\newcommand\norm[1]{\left\lVert#1\right\rVert}

\theoremstyle{plain}
\newtheorem{theorem}{Theorem}[section]

\newtheorem{lemma}[theorem]{Lemma}

\newtheorem{assumptions}[theorem]{Assumptions}

\theoremstyle{definition}
\newtheorem{definition}[theorem]{Definition}

\newtheorem{test}[theorem]{Test}
\def\bhyp#1{\begin{equation}\label{#1}\begin{array}{c}}
\def\ehyp{\end{array}\end{equation}}

\newcounter{cst}

\theoremstyle{remark}
\newtheorem{remark}[theorem]{Remark}

\numberwithin{equation}{section}
\numberwithin{figure}{section}

\newcommand{\RR}{{\mathbb R}}
\newcommand{\NN}{{\mathbb N}}

\newcommand{\cS}{{\mathbb S}}

\def\O{\Omega}
\def\dsp{\displaystyle}

\def\disc{{\mathcal D}}

\def\dr{\partial}
\newcommand{\cK}{{\mathcal K}}

\def\bvarphi{\boldsymbol{\phi}}

\newcommand{\x}{\pmb{x}}

\def\cA{\mathcal{A}}

\def\bphi{\boldsymbol{\varphi}}

\def\A{\mathbb{A}}
\def\B{\mathbb{B}}
\def\cM{{\mathbb M}}

\def\cT{\mathbb T}

\def\cR{\mathcal R}

\newif\ifcorr\corrtrue
\usepackage[normalem]{ulem}
\definecolor{violet}{rgb}{0.580,0.,0.827}

\def\bpsi{{\boldsymbol \psi}}

\def\bw{{\boldsymbol w}}

\newcommand{\ud}{\, \mathrm{d}} 
\def\div{\mathop{\rm div}}

\title[Error Estimates of RDS with Constraints]{Error Estimates of Generic Discretisation of Reaction-Diffusion System with Constraints}

\author{Yahya Alnashri}
\address[Yahya Alnashri]{Department of Mathematics, Al-Qunfudah University College, Umm Al-Qura University, Saudi Arabia}
\email{yanashri@uqu.edu.sa}


\date{\today}

\begin{document}

\begin{abstract}
In this paper, we study a parabolic reaction–diffusion system with constraints that model biofilm growth. Within a unified framework encompassing multiple numerical schemes, we derive the first general convergence rates for approximating this model using both conforming and non-conforming discretisation methods. Under standard assumptions on the time discretisation, we establish the existence and uniqueness of the discrete solution. Numerical experiments are conducted using a mixed finite volume scheme that fits within the proposed unified framework. A test case with an analytical solution is designed to confirm our theoretical convergence rates.
\end{abstract}

\maketitle

\section{Introduction}
\label{sec-intro}
The model studied in this work describes biofilm growth at the porescale, explained in \cite{biofilm-apps-2016}. Let $\O$ be a bounded, open, and connected subset of $\RR^d$ (with $d\geq 1$) with a boundary $\dr\O$. We seek a pair $(\bar p,\bar q)$ satisfying the following system of two coupled non-linear partial differential equations (PDEs) subject to constraints: 
\begin{subequations}\label{pvi-obs} 
\begin{align}
(\partial_t \bar p - \div(\A \nabla \bar p)-f(\bar p,\bar q))(\bar p - \chi ) = 0 &\mbox{\quad in $\O_T:=\O\times (0,T)$,} \label{pvi-obs1}\\
\partial_t \bar p - \div(\A \nabla \bar p) \geq f(\bar p,\bar q) &\mbox{\quad in $\O_T$,} \label{pvi-obs2}\\
\bar p \geq \chi &\mbox{\quad in $\O_T$,} \label{pvi-obs3}\\
\partial_t \bar q - \div(\B \nabla \bar q) = g(\bar p, \bar q) &\mbox{\quad in $\O_T$,} \label{pvi-obs4}\\
(\bar p, \bar q) = (0,0) &\mbox{\quad on $(\partial\O \times (0,T))^2$,} \label{pvi-obs5}\\
(\bar p(\x,0), \bar q(\x,0)) = (p_0, q_0) &\mbox{\quad in $(\Omega \times \{0\})^2$} \label{pvi-obs6},
\end{align}
\end{subequations}
where precise assumptions on the data will be presented in the next section. The primary goal of this paper is to establish, for the first time (to our knowledge), general error estimates for approximating the model \eqref{pvi-obs} using both conforming and non-conforming methods.

Mathematical theories addressing the existence, uniqueness, and stability of systems of partial differential equations (PDEs) have been extensively developed in the literature; see, for instance, \cite{kinderlehrer-2000,glowinski-1981} for linear variational inequalities, \cite{E-1} for non-linear ones, \cite{18,19} for non-linear PDEs.

Numerous studies have also investigated the numerical approximation of parabolic problems. The subproblem \eqref{pvi-obs1}--\eqref{pvi-obs3} with $f=0$ (the parabolic obstacle problem) has been analysed using the upwind implicit finite volume scheme in \cite{obstacle-Eymard}, Crouzeix–Raviart finite element method in \cite{Crouzeix-2020}, and mixed finite volume scheme in \cite{obstacle-alnashri}. Reaction-diffusion equations have been discretised using finite difference schemes \cite{Beatrice-23}, finite element methods \cite{Thomas-2015}, and mixed finite volume method \cite{Y-H-2022}. In \cite{yahya-2025}, we introduced error estimates for the non-conforming approximation of such a system. A finite volume scheme for the system  incorporating a non-linear diffusion operator was studied in \cite{REF-2}. A mathematical analysis of a biofilm model, based on the theory of Navier–Stokes variational inequalities, is presented in \cite{Gokieli-2018} while numerical discretisations of the biofilm model can be found in \cite{Rangarajan-2018,Duvnjak-2016}. However, these mentioned works address the systems without the constraints \eqref{pvi-obs2} and \eqref{pvi-obs3}.

Further, the numerical analysis of the model \eqref{pvi-obs} was explored in several works. For instance, \cite{REF-1} proposed the $\mathbb P1$ finite elements method for the model \eqref{pvi-obs} and proved its convergence order in an appropriate norm. In our earlier work \cite{Alnashri-2022}, we developed a general numerical analysis of the problem \eqref{pvi-obs} (without simulations), proving that the scheme converges along a subsequence of discrete solutions to a weak continuous solution. Nonetheless, no convergence rates or numerical experiments were provided.

To the best of our knowledge, the numerical analysis of systems of semilinear partial differential equations with constraints inclusions remains largely unexplored. This work aims to fill that gap by providing the first general error estimates for approximating the model \eqref{pvi-obs}. The main novelty of our work lies in using the gradient discretisation framework, which allows us to establish generic convergence rates applicable to both conforming and non-conforming schemes, rather than focusing on a specific numerical method.

The paper is organised as follows. Section \ref{sec-weak-disc} introduces the variational formulation of the model together with the gradient discretisation framework. Section \ref{sec-results} presents and proves the main results, including the existence and uniqueness of the approximate solution, and error estimates. We develop a technique to handle the non-linearity arising from inequalities and reaction functions. Section \ref{sec-numerical} provides the first numerical simulation of the model \eqref{pvi-obs} using a non-conforming method known as the HMM scheme. We construct a test case with an exact solution to precisely evaluate the derived convergence rates.


\section{Variational formulation and discrete scheme}\label{sec-weak-disc}

\begin{assumptions}\label{assump-1}
We impose the following assumptions on the data:
\begin{itemize}
\item $\O \subset \RR^d\; (d\geq 1)$ is an open, bounded, and connected set with a $C^2$--regular boundary $\dr\O$, and $T>0$,
\item $\A, \B: \O \to \mathbb S_d(\RR)$ ($d \times d$ matrices) are measurable functions, such that, for a.e. $\A$ and $\B$ are symmetric with eigenvalues in $[d_1,d_2]$,
\item the barrier function $\chi \in H^1(\O)\cap C(\overline\O)$ and is a non-positive on $\dr\O$,  
\item the non-linear functions $f$ and $g$ are Lipschitz continuous on $\RR^2$ with Lipschitz constants $M_1$ and $M_2$, respectively, in which $M=\max\{M_1,M_2\}$,
\item the initial solutions $(p_0,q_0)\in (W^{2,\infty}(\O)\cap \cK) \times W^{2,\infty}(\O)$, where $\cK$ is a closed convex set defined by 
\begin{equation}\label{eq-set}
\cK:=\{ \varphi \in H_0^1(\O)\; : \; \varphi \geq \chi(t)\; \mbox{ in } \O \}.
\end{equation}
\end{itemize}
\end{assumptions}

Let us define the following dependent time closed convex set
\begin{equation}\label{eq-time-set}
\mathbb K:=\{\varphi \in L^2(0,T;H_0^1(\O))\;:\; \varphi(\x,t)\in \cK \mbox{ for a.e. }t\in [0,T]\}.
\end{equation}

Under the above assumptions, we can express the model \eqref{pvi-obs} in a weak sense (variational formulation): Find $(\bar p,\bar q) \in (\cK\cap C^0([0,T];L^2(\O))) \times L^2(0,T;H_0^1(\O))$, such that $(\dr_t\bar p,\dr_t\bar q)\in L^2((0,T;H^{-1}(\O)))^2$, and stisfies the following system:
\begin{subequations}\label{pvi-obs-weak}
\begin{equation}\label{pvi-obs-w1}
\begin{aligned}
&\dsp\int_0^T\langle \partial_t\bar p(\x,t),\varphi(\x,t) \rangle \ud t
+\dsp\int_0^T\int_\O \A\nabla\bar p \cdot \nabla(\bar p-\varphi)(\x,t)\ud \x \ud t\\
&\leq\dsp\int_0^T \int_\O f(\bar p,\bar q)(\bar p(\x,t)-\varphi(\x,t))\ud \x \ud t,\quad \forall \varphi \in \mathbb K,\mbox{ and }
\end{aligned}
\end{equation}
\begin{equation}\label{pvi-obs-w2}
\left.
\begin{aligned}
&\dsp\int_0^T \langle \partial_t\bar q(\x,t),\psi(\x,t) \rangle \ud t
+\dsp\int_0^T\int_\O \B(\x)\nabla \bar q(\x,t) \cdot \nabla \psi(\x,t)\ud \x \ud t\\
&\quad= \dsp\int_0^T\dsp\int_\O g(\bar p(\x,t),\bar q(\x,t))\psi(\x,t) \ud \x \ud t,\quad \forall \psi \in L^2(0,T;H_0^1(\O)),
\end{aligned}
\right.
\end{equation}
\end{subequations}
where $\langle \cdot,\cdot \rangle$ is the duality product between the spaces $H^{-1}(\O)$ and $H^1(\O)$. The existence of a weak solution to this problem will be a result of the unified analysis provided here, see Remark \ref{remark-1}.
 
\begin{definition}[{\bf Discrete elements}]\label{def-gd-pvi-obs} A gradient discretisation $\disc$ is defined by $\disc=(X_{\disc,0},\Pi_\disc,\nabla_\disc, \chi_\disc, P_\disc, I_\disc, J_\disc)$, where:
\begin{enumerate}
\item The discrete set $X_{\disc,0}$ is a finite-dimensional vector space on $\RR$, taking into account the homogenous Dirichlet boundary conditions.
\item The linear operator $\Pi_\disc : X_{\disc,0} \to L^2(\O)$ is the reconstruction of the approximate function.
\item The linear operator $\nabla_\disc : X_{\disc,0} \to L^2(\O)^d$ is the reconstruction of the gradient of the function, and must be chosen so that $\| \nabla_\disc\cdot \|_{L^2(\O)^d}$ is a norm on $X_{\disc,0}$.
\item $\chi_\disc\in L^2(\O)$ is an approximation of the barrier $\chi$.
\item $P_\disc:H_0^1(\O) \cap H^2(\O) \to X_{\disc,0}$ is a linear continuous interpolant and must be constructed so that $P_\disc(\cK\cap H^2(\O))\subset \cK_\disc$, where
\begin{equation}\label{eq-KD}
\cK_\disc:=\{ \varphi\in X_{\disc,0} :\; \Pi_\disc \varphi\geq \chi_\disc \mbox{ in } \O \}.
\end{equation}
\item $I_\disc: W^{2,\infty}(\O)\cap\cK \to \cK_\disc$ is a linear and continuous interpolation operator for the initial solution $p_0$,
\item $J_\disc: W^{2,\infty}(\O) \to X_{\disc,0}$ is a linear and continuous interpolation operator for the solution $q_0$.
\end{enumerate}
\end{definition}

\begin{remark}
In our discretisation, we introduce an approximate barrier $\chi_\disc$ in the definition of the discrete set $\cK_\disc$ to address the challenge of constructing an interpolant that satisfies the barrier throughout the entire space $\O$. It is worth noting that most numerical schemes define interpolants solely from the solution values $\bar p$ at mesh vertices. This approach does not necessarily guarantee that the constraint  \eqref{pvi-obs3} is satisfied at all spatial points, except in the special case where the barrier is constant.
\end{remark}

To develop an effective approximate scheme, only three core properties are required: coercivity, consistency, and limit-conformity. These properties are governed by a set of parameters and functions defined within our framework,
\begin{equation}\label{corc-eq}
C_\disc = \dsp\max_{\varphi \in X_{\disc,0}\setminus\{0\}}\frac{\|\Pi_\disc \varphi\|_{L^2(\O)}}{\|\nabla_\disc \varphi\|_{L^2(\O)^d}},
\end{equation}

\begin{equation}\label{consist-1}
\begin{aligned}
&S_\disc:H_0^1(\O) \times X_{\disc,0} \to [0,\infty), \mbox{ which is defined by}\\
&\forall (w,\varphi) \in H_0^1(\O) \times X_{\disc,0}, \; S_\disc(w,\varphi)=\| \Pi_\disc \varphi - w \|_{L^2(\O)} 
+ \| \nabla_\disc \varphi - \nabla w \|_{L^2(\O)^{d}},
\end{aligned}
\end{equation}
\begin{equation}\label{long-rm}
\begin{aligned}
&W_{\mathcal{D}} : H_{\rm div}(\O):=\{\bpsi \in L^2(\O)^d\;:\; {\rm div}\bpsi \in L^2(\O)\}\to [0, +\infty), \mbox{ which is defined by}\\
&W_{\mathcal{D}}(\bpsi)
 = \sup_{\varphi\in X_{\disc,0}\setminus \{0\}}\frac{\Big|\dsp\int_{\O}(\nabla_\disc \varphi\cdot \bpsi + \Pi_\disc \varphi \div (\bpsi)) \ud \x \Big|}{\| \nabla_\disc \varphi\|_{L^2(\O)^d} }.
\end{aligned}
\end{equation}

\begin{remark}
Unlike our analysis in \cite{Alnashri-2022}, which employs two distinct parameters to assess the scheme's consistency, we introduce here a single parameter $S_\disc$ to control the interpolation error across elements $\bar q,\dr_t \bar q, \dr_t \bar p \in H_0^1(\O)$ and $\bar p \in \cK$. Also, we note that Definition \ref{def-gd-pvi-obs} involves a linear spatial interpolator $P_\disc:\cK \cap H^2(\O)\to \cK_\disc$ while \cite[Lemma 1]{droniou-2017-error} proves that there exists a linear one $\widetilde P_\disc:H_0^1(\O)\to X_{\disc,0}$ defined by
\begin{equation}\label{PD}
\widetilde P_\disc(w):=\arg\min_{w\in X_{\disc,0}} \|\Pi_\disc w - \varphi\|_{L^2(\O)}
+\|\nabla_\disc w - \nabla\varphi\|_{L^2(\O)}.
\end{equation}
Thanks to the definition of $S_\disc$, we obtain
\begin{equation}\label{new-eq-proof-SD-u}
\|\Pi_\disc P_\disc\varphi - \varphi\|_{L^2(\O)}
+\|\nabla_\disc P_\disc\varphi - \nabla\varphi\|_{L^2(\O)^d}
\leq S_\disc(\varphi,P_\disc\varphi),\quad \forall \varphi \in H^2(\O)\cap \cK.
\end{equation}
and
\begin{equation}\label{new-eq-proof-100}
\|\Pi_\disc \widetilde P_\disc\varphi - \varphi\|_{L^2(\O)}
+\|\nabla_\disc \widetilde P_\disc\varphi - \nabla\varphi\|_{L^2(\O)^d}
\leq S_\disc(\varphi,\widetilde P_\disc\varphi),\quad \forall \varphi \in H_0^1(\O).
\end{equation}

It should also be emphasised that the interpolant $P_\disc$ is required only for two elements, $\bar p$ and $\dr_t\bar p$, and identifying a suitable interpolant for a specific numerical scheme is generally straightforward. We refer the reader to \cite[Remark 3.2 and Section 4]{obstacle-alnashri} for details.
\end{remark}

In what follows, let $N\in \NN^\star$ and $0\leq n \leq N-1$. We define a discretisation of the time interval $[0,T]$ by $t^{(0)}=0<...<t^{(N)}=T$, with the time step $\delta t^{(n+\frac{1}{2})}:=t^{(n+1)}-t^{(n)}$. We set $\delta t:=\max_{n} \delta t^{(n+\frac{1}{2})}$. The discrete derivative $\delta_\disc \varphi \in L^\infty(0,T;L^2(\O))$ of $\varphi\in X_{\disc,\Gamma_2}^{N+1}$ is defined by  
\begin{equation*}
\delta_\disc \varphi(t)=\delta_\disc^{(n+\frac{1}{2})}\varphi:=\frac{\Pi_\disc \varphi^{(n+1)}-\Pi_\disc \varphi^{(n)}}{\delta t^{(n+\frac{1}{2})}}, \mbox{ $\forall n=0,...,N-1$ and $t\in (t^{(n)},t^{(n+1)}]$}.
\end{equation*}

Using the gradient discretisation in Definition \ref{def-gd-pvi-obs}, we can introduce a generic implicit Euler scheme for the problem \eqref{pvi-obs-weak}, it is named a gradient scheme. Seek $(p^{(n)},q^{(n)})_{n=0,...,N} ) \subset \cK_\disc \times X_{\disc,0}$, such that $(p^{(0)},q^{(0)})=( J_\disc p_0, \widetilde J_\disc q_0 )\in \cK_\disc \times X_{\disc,0}$, for all $n=0,...,N-1$, the following inequality and equality hold:
\begin{subequations}\label{gs-pvi}
\begin{equation}\label{gs-pvi-obs1}
\begin{array}{ll}
\dsp\int_\O \delta_\disc^{(n+\frac{1}{2})}p(\x)\, \Pi_\disc(p^{(n+1)}-\varphi)(\x) \ud \x+\dsp\int_\O \A(\x)\nabla_\disc p^{(n+1)}(\x)\cdot\nabla_\disc(p^{(n+1)}-\varphi)(\x) \ud \x\\
\leq \dsp\int_\O f(\Pi_\disc p^{(n+1)}, \Pi_\disc q^{(n+1)})\Pi_\disc(p^{(n+1)}-\varphi)(\x) \ud \x,\quad \forall \varphi \in \cK_\disc, \mbox{ and } 
\end{array}
\end{equation}
\begin{equation}\label{gs-pvi-obs2}
\begin{array}{ll}
\dsp\int_\O \delta_\disc^{(n+\frac{1}{2})}v(\x)\, \Pi_\disc \psi(\x) \ud \x
+\dsp\int_\O \B(\x)\nabla_\disc q^{(n+1)}(\x)\cdot\nabla_\disc \psi(\x) \ud \x\\
= \dsp\int_\O g(\Pi_\disc p^{(n+1)}, \Pi_\disc q^{(n+1)})\Pi_\disc\psi(\x) \ud \x,\quad \forall \psi \in X_{\disc,0}.
\end{array}
\end{equation}
\end{subequations}

\section{Main Results}\label{sec-results}
We begin by establishing the existence and uniqueness of the approximate solution. We develop a similar technique as in \cite{BRADJI-2016}(which only applies to a single equality) to deal with the system and the inequalities.

\begin{lemma}\label{lemma-1}
Assume that Assumptions \ref{assump-1} hold and let $\disc$ be a gradient discretisation. If $\delta t^{(n+\frac{1}{2})}< \frac{1}{2M}$, then there exists a unique solution to the scheme \eqref{gs-pvi}.
\end{lemma}

\begin{proof}
At any time step $n+1$, assume $p^{(n)}$ and $q^{(n)}$ exist and are unique. The scheme expresses a non-linear square system, whose unknowns are $p^{(n+1)}$ and $q^{(n+1)}$. For any $\bw=(w_1,w_2) \in \Pi_\disc(\cK_\disc) \times \Pi_\disc(X_{\disc,0})$, we can find a unique pair $(U_1,U_2)\in \cK_\disc \times X_{\disc,0}$ satisying
\begin{equation}\label{ex-u-1}
\begin{aligned}
&\frac{1}{\delta t^{(n+\frac{1}{2})}}\dsp\int_\O \Pi_\disc(U_1-p^{(n)})(\x)\Pi_\disc(U_1- \varphi)(\x)\ud \x
+\dsp\int_\O \A(\x)\nabla_\disc p(\x) \cdot \nabla_\disc(U_1- \varphi)(\x)\ud \x\\
&\qquad\leq\dsp\int_\O f(w_1,w_2)\Pi_\disc(U_1-\varphi)(\x) \ud \x, \quad \forall \varphi \in \cK_\disc,\mbox{ and}
\end{aligned}
\end{equation}
\begin{equation}\label{ex-v-1}
\begin{aligned}
&\frac{1}{\delta t^{(n+\frac{1}{2})}}\dsp\int_\O \Pi_\disc(U_2-q^{(n)})(\x)\Pi_\disc \psi(\x) \ud \x
+\dsp\int_\O \B(\x)\nabla_\disc q(\x) \cdot \nabla_\disc \psi(\x)\ud \x\\
&\qquad=\dsp\int_\O g(w_1,w_2)\Pi_\disc \psi(\x) \ud \x, \quad \forall \psi \in X_{\disc,0}.
\end{aligned}
\end{equation}
Let us define the mapping $\cT:\Pi_\disc(\cK_\disc) \times \Pi_\disc(X_{\disc,0}) \to \Pi_\disc(\cK_\disc) \times \Pi_\disc(X_{\disc,0})$, such that $\cT(\bw)=(\Pi_\disc p,\Pi_\disc q)$, for a given $\bw=(w_1,w_2) \in \times \Pi_\disc(\cK_\disc) \times \Pi_\disc(X_{\disc,0})$, and $(U_1,U_2)$ satisfies \eqref{ex-u-1}--\eqref{ex-v-1}. Given that $(p^{(0)},q^{(0)})$ exists and is unique, we only need to show that $\cT$ is a contractive mapping to establish the existence and uniqueness of the approximate solution $(p,q)$ satisfying \eqref{gs-pvi}.

Let $\bw=(w_1,w_2)$, $\tilde\bw=(\tilde w_1,\tilde w_2) \in \Pi_\disc(\cK_\disc)\times \Pi_\disc(X_{\disc,0})$, such that $\cT(\bw)=(\Pi_\disc p,\Pi_\disc q)$ and $\cT(\tilde\bw)=(\Pi_\disc \tilde U_1,\Pi_\disc \tilde U_2)$. From \eqref{ex-u-1} and \eqref{ex-v-1}, we have
\begin{equation}\label{ex-u-1-a}
\begin{aligned}
&\frac{1}{\delta t^{(n+\frac{1}{2})}}\dsp\int_\O \Pi_\disc(U_1-p^{(n)})(\x)\Pi_\disc (U_1-\varphi)(\x)\ud \x
+\dsp\int_\O \A(\x)\nabla_\disc p(\x) \cdot \nabla_\disc(U_1-\varphi)(\x)\ud \x\\
&\qquad\leq\dsp\int_\O f(w_1, w_2)\Pi_\disc(U_1-\varphi)(\x) \ud \x, \quad \forall \varphi \in \cK_\disc,
\end{aligned}
\end{equation}
\begin{equation}\label{ex-u-1-b}
\begin{aligned}
&\frac{1}{\delta t^{(n+\frac{1}{2})}}\dsp\int_\O \Pi_\disc(\tilde U_1-p^{(n)})(\x)\Pi_\disc(\tilde U_1-\varphi)(\x) \ud \x
+\dsp\int_\O \A(\x)\nabla_\disc \tilde U_1(\x) \cdot \nabla_\disc(\tilde U_1-\varphi)(\x)\ud \x\\
&\qquad\leq\dsp\int_\O f(\tilde w_1, \tilde w_2)\Pi_\disc(\tilde U_1-\varphi)(\x) \ud \x, \quad \forall \varphi \in \cK_\disc,
\end{aligned}
\end{equation}
\begin{equation}\label{ex-v-1-a}
\begin{aligned}
&\frac{1}{\delta t^{(n+\frac{1}{2})}}\dsp\int_\O \Pi_\disc(U_2-q^{(n)})(\x)\Pi_\disc \psi(\x)
+\dsp\int_\O \B(\x)\nabla_\disc q(\x) \cdot \nabla_\disc \psi(\x)\ud \x\\
&\qquad=\dsp\int_\O g(w_1, w_2)\Pi_\disc \psi(\x) \ud \x, \quad \forall \psi \in X_{\disc,0},\mbox{ and }
\end{aligned}
\end{equation}
\begin{equation}\label{ex-v-1-b}
\begin{aligned}
&\frac{1}{\delta t^{(n+\frac{1}{2})}}\dsp\int_\O \Pi_\disc(\tilde U_2-q^{(n)})(\x)\Pi_\disc \psi(\x)
+\dsp\int_\O \B(\x)\nabla_\disc  \tilde U_2(\x) \cdot \nabla_\disc \psi(\x)\ud \x\\
&\qquad=\dsp\int_\O g(\tilde w_1, \tilde w_2)\Pi_\disc \psi(\x) \ud \x, \quad \forall \psi \in X_{\disc,0}.
\end{aligned}
\end{equation}
Take $\varphi:=\tilde U_1$ in \eqref{ex-u-1-a} and $\varphi:=U_1$ in \eqref{ex-u-1-b} to obtain
\begin{equation}\label{new-ex-u-1-a}
\begin{aligned}
&\frac{1}{\delta t^{(n+\frac{1}{2})}}\dsp\int_\O \Pi_\disc(p^{(n)}-U_1)(\x)\Pi_\disc (\tilde U_1-U_1)(\x)\ud \x
-\dsp\int_\O \A(\x)\nabla_\disc p(\x) \cdot \nabla_\disc(\tilde U_1-U_1)(\x)\ud \x\\
&\qquad\leq-\dsp\int_\O f(w_1, w_2)\Pi_\disc(\tilde U_1-U_1)(\x) \ud \x,\mbox{ and }
\end{aligned}
\end{equation}
\begin{equation}\label{new-ex-u-1-b}
\begin{aligned}
&\frac{1}{\delta t^{(n+\frac{1}{2})}}\dsp\int_\O \Pi_\disc(\tilde U_1-p^{(n)})(\x)\Pi_\disc(\tilde U_1-U_1)(\x) \ud \x
+\dsp\int_\O \A(\x)\nabla_\disc \tilde U_1(\x) \cdot \nabla_\disc(\tilde U_1-U_1)(\x)\ud \x\\
&\qquad\leq\dsp\int_\O f(\tilde w_1, \tilde w_2)\Pi_\disc(\tilde U_1-U_2)(\x) \ud \x.
\end{aligned}
\end{equation}
Adding Inequality \eqref{new-ex-u-1-a} to Inequality \eqref{new-ex-u-1-b} results in
\begin{equation}\label{new-ex-u-2-a}
\begin{aligned}
&\frac{1}{\delta t^{(n+\frac{1}{2})}}\dsp\int_\O \Pi_\disc(\tilde U_1-U_1)(\x)\Pi_\disc (\tilde U_1-U_1)(\x)\ud \x
+\dsp\int_\O \A(\x)\nabla_\disc (\tilde U_1-U_1)(\x) \cdot \nabla_\disc(\tilde U_1-U_1)(\x)\ud \x\\
&\qquad\leq\dsp\int_\O (f(\tilde w_1, \tilde w_2)-f(w_1, w_2))\Pi_\disc(\tilde U_1-U_1)(\x) \ud \x.
\end{aligned}
\end{equation}
Subtracting Equation \eqref{ex-v-1-b} from Equation \eqref{ex-v-1-a} leads to
\begin{equation}\label{ex-v-2}
\begin{aligned}
&\frac{1}{\delta t^{(n+\frac{1}{2})}}\dsp\int_\O \Pi_\disc(U_2-\tilde U_2)(\x)\Pi_\disc \psi(\x)
+\dsp\int_\O \B(\x)\nabla_\disc (U_2-\tilde U_2)(\x) \cdot \nabla_\disc \psi(\x)\ud \x\\
&\qquad=\dsp\int_\O \left(g(w_1, w_2)
-g(\tilde w_1, \tilde w_2) \right)\Pi_\disc \psi(\x) \ud \x,\quad \forall \psi \in X_{\disc,0}.
\end{aligned}
\end{equation}
Set $\psi=U_2-\tilde U_2\in X_{\disc,0}$ in \eqref{ex-v-2}, together with \eqref{new-ex-u-1-b}, to derive
\begin{equation}\label{ex-u-3}
\begin{aligned}
&\frac{1}{\delta t^{(n+\frac{1}{2})}}\|\Pi_\disc(U_1-\tilde U_1)\|_{L^2(\O)}^2
+d_1\|\nabla_\disc(U_1-\tilde U_1)\|_{L^2(\O)^d}^2\\
&\qquad\leq \dsp\int_\O \left(f(w_2, w_1)
-f(\tilde w_1, \tilde w_2) \right)\Pi_\disc(U_1-\tilde U_1)(\x) \ud \x, \mbox{ and }
\end{aligned}
\end{equation}
\begin{equation}\label{ex-v-3}
\begin{aligned}
&\frac{1}{\delta t^{(n+\frac{1}{2})}}\|\Pi_\disc(U_2-\tilde U_2)\|_{L^2(\O)}^2
+d_1\|\nabla_\disc(U_2-\tilde U_2)\|_{L^2(\O)^d}^2\\
&\qquad\leq \dsp\int_\O \left(g(w_1, w_2)
-g(\tilde w_1, \tilde w_2) \right)\Pi_\disc(U_2-\tilde U_2)(\x) \ud \x.
\end{aligned}
\end{equation}
From the Cauchy–Schwarz inequality and the Lipschitz continuity conditions on $f$ and $g$, bounds on the right-hand sides of the above inequalities can be established as follows:
\begin{equation}\label{ex-u-4}
\begin{aligned}
&\dsp\int_\O \left(f(w_1, w_2)
-f(\tilde w_1, \tilde w_2) \right)\Pi_\disc(U_1-\tilde U_1)(\x) \ud \x\\
&\leq M\left(\|w_1 - \tilde w_1\|_{L^2(\O)} + \|w_2 - \tilde w_2\|_{L^2(\O)^d}  \right)\|\Pi_\disc(U_1-\tilde U_1)\|_{L^2(\O)},\mbox{ and}
\end{aligned}
\end{equation}
\begin{equation}\label{ex-v-4}
\begin{aligned}
&\dsp\int_\O \left(g(w_1, w_2)
-g(\tilde w_1, \tilde w_2) \right)\Pi_\disc(U_2-\tilde U_2)(\x) \ud \x\\
&\leq M\left(\|w_1 - \tilde w_1\|_{L^2(\O)} + \|w_2 - \tilde w_2\|_{L^2(\O)^d}  \right)\|\Pi_\disc(U_2-\tilde U_2)\|_{L^2(\O)}.
\end{aligned}
\end{equation}
injecting these relations into \eqref{ex-u-3} and into \eqref{ex-v-3} gives
\begin{equation*}
\begin{aligned}
\frac{1}{\delta t^{(n+\frac{1}{2})}}\|\Pi_\disc(U_1-\tilde U_1)\|_{L^2(\O)}
\quad\leq M\left(\|w_1 - \tilde w_1\|_{L^2(\O)} + \|w_2 - \tilde w_2\|_{L^2(\O)}\right), \mbox{ and }
\end{aligned}
\end{equation*}
\begin{equation*}
\begin{aligned}
\frac{1}{\delta t^{(n+\frac{1}{2})}}\|\Pi_\disc(U_2-\tilde U_2)\|_{L^2(\O)}
\quad\leq M\left(\|w_1 - \tilde w_1\|_{L^2(\O)} + \|w_2 - \tilde w_2\|_{L^2(\O)}\right),
\end{aligned}
\end{equation*}
which lead to
\begin{equation*}\label{ex-6}
\begin{aligned}
\frac{1}{\delta t^{(n+\frac{1}{2})}}[\|\Pi_\disc(U_1-\tilde U_1)\|_{L^2(\O)}
&+\|\Pi_\disc(U_2-\tilde U_2)\|_{L^2(\O)} ]\\
&\leq 2M \left(\|w_1 - \tilde w_1\|_{L^2(\O)} + \|w_2 - \tilde w_2\|_{L^2(\O)}\right),
\end{aligned}
\end{equation*}
Since $\cT(\bw)=(\Pi_\disc p,\Pi_\disc q)$ and $\cT(\tilde\bw)=(\Pi_\disc \tilde U_1,\Pi_\disc \tilde U_2)$, the above inequality yields the following relation
\begin{equation*}
\begin{aligned}
\frac{1}{\delta t^{(n+\frac{1}{2})}}\|g(w_1,w_2)-g(\tilde w_1,\tilde w_2)\|_{L^2(\O)}
\leq 2L\|(w_1,w_2)-(\tilde w_1,\tilde w_2)\|_{L^2(\O)},
\end{aligned}
\end{equation*}
which proves that $\cT$ is a contraction mapping on $\Pi_\disc(\cK_\disc) \times \Pi_\disc(X_{\disc,0})$ provided that $\delta t^{(n+\frac{1}{2})}< \frac{1}{2M}$ and it has a unique fixed point $(w_1,w_2)=\cT(w_1,w_2)$. This shows that the existence and uniqueness of the pair $(\Pi_\disc p,\Pi_\disc q)$, such that $(U_1,U_2)$ solves the problem \eqref{gs-pvi}. Using the relations $\Pi_\disc p=\Pi_\disc \tilde U_1 $ in \eqref{ex-u-3} and the relation $\Pi_\disc q=\Pi_\disc \tilde U_2$ in \eqref{ex-v-3}, we obtain $(\nabla_\disc p,\nabla_\disc q)= (\nabla_\disc \tilde U_1=\nabla_\disc \tilde U_2)$, which completes the proof.
\end{proof}

\subsection*{Notations}
In what follows, we use the following notations for simplicity.
\begin{subequations}\label{eq-def-S}
\begin{equation}\label{eq-def-S-1}
\cS(\varphi):=S_\disc(\varphi,P_\disc \varphi),\mbox{ for all } \varphi=\bar p \mbox{ or $\dr_t\bar p$, and }
\end{equation}
\begin{equation}\label{eq-def-S-1}
\widetilde\cS(\varphi):=S_\disc(\varphi,\widetilde P_\disc \varphi),\mbox{ for all } \varphi=\bar q \mbox{ or $\dr_t\bar q$}.
\end{equation}
\end{subequations}
Also, we denote the error coming from the interpolation of the initial conditions by $\cR_\disc^0$ and $\widetilde \cR_\disc^0$, which are defined by
\begin{equation}\label{eq-err-inter-0}
\cR_\disc^0=\|p_0 - \Pi_\disc J_\disc p_0 \|_{L^2(\O)}\mbox{ and } 
\widetilde \cR_\disc^0=\|q_0 - \Pi_\disc \widetilde J_\disc q_0 \|_{L^2(\O)}.
\end{equation}
We set
\begin{equation}\label{eq-E}
\cR^{(k)}:=P_\disc\bar p(t^{(k)})-p^{(k)} \mbox{ and } \widetilde \cR^{(k)}:=\widetilde P_\disc\bar q(t^{(k)})-q^{(k)}, \mbox{ for } k=1,...,N.
\end{equation}
Finally, with letting $(\bar p^{(0)},\bar q^{(0)})=(\bar p(0),\bar q(0))$, the notation $\zeta^{(n+1)}$ denotes the averaging over time in the interval $(t^{(n)}, t^{(n+1)})$, in which $\xi=f,g,\bar p, \bar q,\dr_t\bar p,$ or $\dr_t \bar q$. It defined by, for $n\in \{0,...,N-1\}$,
\[
\xi^{(n+1)}(\x):=\dsp\frac{1}{\delta t^{(n+\frac{1}{2})}}\dsp\int_{t^{(n)}}^{t^{(n+1)}}
\zeta(\x,t) \ud t.
\]

\begin{lemma}\label{lemma-reaction}
Let Assumptions \ref{assump-1} hold, $(\bar p,\bar q)$ be the solution to the problem \eqref{pvi-obs-weak}, and $(p,q)$ be the solution to the scheme \eqref{gs-pvi}. If $\bar p,\bar q:[0,T] \to H^1(\O)$ are Lipschitz continuous functions, then the following estimates hold:
\begin{equation}\label{eq-proof-F-1}
\begin{aligned}
&\int_\O \Pi_\disc \cR^{(n+1)}(\x)\Big[f(\bar p^{(n+1)},\bar q^{(n+1)})-f(\Pi_\disc p^{(n+1)},\Pi_\disc q^{(n+1)}) \Big] \ud \x\\
&\leq C_1\|\Pi_\disc  \cR^{(n+1)}\|_{L^2(\O)}^2
+C_2\|\Pi_\disc \widetilde \cR^{(n+1)}\|_{L^2(\O)}^2\\
&\quad+C_3\Big[\delta t_\disc+\cS(\bar p(t^{(n+1)}))\Big]^2
+C_4\Big[\delta t_\disc+\widetilde\cS(\bar q(t^{(n+1)}))\Big]^2,\mbox{ and }
\end{aligned}
\end{equation}
\begin{equation}\label{eq-proof-G-1}
\begin{aligned}
&\int_\O \Pi_\disc \widetilde \cR^{(n+1)}(\x)\Big[g(\bar p^{(n+1)},\bar q^{(n+1)})-g(\Pi_\disc p^{(n+1)},\Pi_\disc q^{(n+1)}) \Big] \ud \x\\
&\leq C_5\|\Pi_\disc  \cR^{(n+1)}\|_{L^2(\O)}^2
+C_6\|\Pi_\disc \widetilde \cR^{(n+1)}\|_{L^2(\O)}^2\\
&\quad+C_7 \Big[\delta t_\disc+\cS(\bar p(t^{(n+1)}))\Big]^2
+C_8\Big[\delta t_\disc+\widetilde\cS(\bar q(t^{(n+1)}))\Big]^2,
\end{aligned}
\end{equation}
where $(C_i)_{i=1,...,8}$ depend only on $M$ and small positive parameters linked to the Young's inequality.
\end{lemma}

\begin{proof}
From the Cauchy–Schwarz inequality and the Lipschitz continuity assumption for $f$ and $g$, one can write
\begin{equation}\label{eq-proof-F}
\begin{aligned}
&\int_\O \Pi_\disc \cR^{(n+1)}(\x)\Big[f(\bar p^{(n+1)},\bar q^{(n+1)})-f(\Pi_\disc p^{(n+1)},\Pi_\disc q^{(n+1)}) \Big] \ud \x\\
&\leq M\|\Pi_\disc \cR^{(n+1)}\|_{L^2(\O)}
\|\bar p^{(n+1)}-\Pi_\disc p^{(n+1)}\|_{L^2(\O)}:=\mathbb I_1\\
&\quad+M\|\Pi_\disc \cR^{(n+1)}\|_{L^2(\O)}
\|\bar q^{(n+1)}-\Pi_\disc q^{(n+1)}\|_{L^2(\O)}:=\mathbb I_2.
\end{aligned}
\end{equation}
The application of the interpolants $P_\disc$ to $\varphi:=\bar p(t^{(n+1)})$ and $\widetilde P_\disc$  to $\varphi:=\bar q(t^{(n+1)})$ leads to, thanks to \eqref{new-eq-proof-SD-u}, \eqref{new-eq-proof-100}, and to the Lipschitz continuity of the solution $(\bar p,\bar q)$,
\begin{equation}\label{eq-proof-Pi-u}
\begin{aligned}
\| \bar p^{(n+1)}&-\Pi_\disc  P_\disc \bar p(t^{(n+1)}) \|_{L^2(\O)}\\
&\leq \| \bar p^{(n+1)}-\bar p(t^{(n+1)}) \|_{L^2(\O)} 
+\|\bar p(t^{(n+1)}) -\Pi_\disc  P_\disc \bar p(t^{(n+1)}) \|_{L^2(\O)}\\
&\leq \delta t_\disc+\cS(\bar p(t^{(n+1)})),\mbox{ and }
\end{aligned}
\end{equation}
\begin{equation}\label{eq-proof-Pi-v}
\begin{aligned}
\| \bar q^{(n+1)}&-\Pi_\disc  \widetilde P_\disc \bar q(t^{(n+1)}) \|_{L^2(\O)}\\
&\leq \| \bar q^{(n+1)}-\bar q(t^{(n+1)}) \|_{L^2(\O)}+
\|\bar q(t^{(n+1)}) -\Pi_\disc  \widetilde P_\disc \bar q(t^{(n+1)}) \|_{L^2(\O)}\\
&\leq \delta t_\disc+\widetilde\cS(\bar q(t^{(n+1)})).
\end{aligned}
\end{equation}
We obtain by adding the term $\Pi_\disc P_\disc \bar p(t^{(n+1)})$, using \eqref{eq-proof-Pi-u}, and applying the Young's inequality with a small parameter $\varepsilon_1>0$
\begin{equation}\label{eq-proof-T1}
\begin{aligned}
\mathbb I_1
&\leq M\|\Pi_\disc \cR^{(n+1)}\|_{L^2(\O)}\times\\
&\quad\Big[
\|\bar p^{(n+1)}-\Pi_\disc P_\disc \bar p(t^{(n+1)})\|_{L^2(\O)}
+\|\Pi_\disc P_\disc \bar p(t^{(n+1)})-\Pi_\disc p^{(n+1)}\|_{L^2(\O)}
\Big]\\
&\leq M\|\Pi_\disc \cR^{(n+1)}\|_{L^2(\O)}
\Big[
\delta t_\disc+\cS(\bar p(t^{(n+1)}))
+\|\Pi_\disc \cR^{(n+1)}\|_{L^2(\O)}
\Big]\\
&\leq \frac{M^2\varepsilon_1+2M}{2}\|\Pi_\disc \cR^{(n+1)}\|_{L^2(\O)}^2
+\frac{1}{2\varepsilon_1}\Big[\delta t_\disc+\cS(\bar p(t^{(n+1)}))\Big]^2.
\end{aligned}
\end{equation}
Arguing as those above (with using \eqref{eq-proof-Pi-v}), it follows that
\begin{equation}\label{eq-proof-T2}
\begin{aligned}
\mathbb I_2 &\leq M\|\Pi_\disc \cR^{(n+1)}\|_{L^2(\O)}\times\\
&\quad \Big[
\|\bar q^{(n+1)}-\Pi_\disc P_\disc \bar q(t^{(n+1)})\|_{L^2(\O)}
+\|\Pi_\disc P_\disc \bar q(t^{(n+1)})-\Pi_\disc q^{(n+1)}\|_{L^2(\O)}
\Big]\\
&\leq M\|\Pi_\disc \cR^{(n+1)}\|_{L^2(\O)}
\Big[
\delta t_\disc+\widetilde\cS(\bar q(t^{(n+1)}))
+\|\Pi_\disc \widetilde \cR^{(n+1)}\|_{L^2(\O)}
\Big]\\
&\leq \frac{M^2(\varepsilon_2+\varepsilon_3)}{2}\|\Pi_\disc \cR^{(n+1)}\|_{L^2(\O)}^2
+\frac{1}{2\varepsilon_2}\Big[\delta t_\disc+\widetilde\cS(\bar q(t^{(n+1)}))\Big]^2\\
&\quad+\frac{1}{2\varepsilon_3}\|\Pi_\disc \widetilde \cR^{(n+1)}\|_{L^2(\O)}^2,
\end{aligned}
\end{equation}
where $\varepsilon_2,\varepsilon_3$ are small positive parameters connected to the Young's inequality. Combining \eqref{eq-proof-T1}, \eqref{eq-proof-T2}, and \eqref{eq-proof-F} establishes the first estimate \eqref{eq-proof-F-1}, where $C_1:=\frac{M^2(\varepsilon_1+\varepsilon_2+\varepsilon_3)+2M}{2}$, $C_2:=\frac{1}{2\varepsilon_3}$, $C_3:=\frac{1}{2\varepsilon_1}$, and $C_4:=\frac{1}{\varepsilon_3}$.

With a similar manner, we can prove the second estimate \eqref{eq-proof-G-1}, in which $C_5:=\frac{1}{2\varepsilon_6}$, $C_6:=\frac{M^2(\varepsilon_4+\varepsilon_5+\varepsilon_6)+2M}{2}$, $C_7:=\frac{1}{\varepsilon_6}$, and $C_8:=\frac{1}{2\varepsilon_4}$, where $(\varepsilon_i)_{i=4,5,6}$ are small positive parameters connected to the Young's inequality.
\end{proof}

\begin{theorem}\label{thm-err-rm} 
Let Assumptions \ref{assump-1} be satisfied, $\disc$ be a gradient discretisation, and  $(\bar p,\bar q)$ and $(p,q)$ be the solutions to the continuous problem \eqref{pvi-obs-weak} and to the approximate problem \eqref{gs-pvi}, respectively. If $(\bar p,\bar q) \in (W^{1,\infty}(0,T;W^{2,\infty}(\O)))^2$, $\delta t^{(n+\frac{1}{2})}< \frac{1}{2M}$, and $\A$ and $\B$ are Lipschitz continuous, then the following estimates hold:
\begin{equation}\label{eq-est-1-theorem}
\begin{aligned}
&\max_{t\in[0,T]}\Big[\|\Pi_\disc p(\cdot,t)-\bar p(\cdot,t) \|_{L^2(\O)}
+\|\Pi_\disc q(\cdot,t)-\bar q(\cdot,t) \|_{L^2(\O)}\Big]\\
&\leq C\Big[ \delta t_\disc+\cS(\bar p(t^{(N)}))
+\widetilde\cS(\bar q(t^{(N)}))
+\cS(\dr_t\bar p^{(N)})
+\widetilde\cS(\dr_t\bar q^{(N)})\\
&\quad+W_\disc(\A\nabla\bar p^{(N)})
+W_\disc(\B\nabla\bar q^{(N)})
+\widetilde\cS(\bar q(0)) + \widetilde \cR_\disc^0
+\cS(\bar p(0)) + \cR_\disc^0\\
&\quad+\Big(\dsp\sum_{n=0}^{N-1} \delta t^{(n+\frac{1}{2})}\cM_\disc^{(n+1)}\Big)^{1/2}
\Big],
\end{aligned}
\end{equation}
and
\begin{equation}\label{eq-est-2-theorem}
\begin{aligned}
&\|\nabla_\disc p-\nabla\bar p \|_{L^2(\O)^d}
+\|\nabla_\disc q-\nabla\bar q \|_{L^2(\O)^d}\\
&\leq C\dsp\sum_{n=0}^{N-1} \delta t^{(n+\frac{1}{2})}\Big[ \delta t_\disc+\cS(\bar p(t^{(n+1)}))
+\widetilde\cS(\bar q(t^{(n+1)}))
+\cS(\dr_t\bar p^{(N)})
+\widetilde\cS(\dr_t\bar q^{(n+1)})\\
&\quad+W_\disc(\A\nabla\bar p^{(n+1)})
+W_\disc(\B\nabla\bar q^{(n+1)})
+\widetilde\cS(\bar q(0)) + \widetilde \cR_\disc^0
+\cS(\bar p(0) + \cR_\disc^0\\
&\quad+(\cM_\disc^{(n+1)})^{1/2}
\Big],
\end{aligned}
\end{equation}
where the term $\cM_\disc^{(n+1)}$ is defined by
\begin{equation}\label{eq-M}
\begin{aligned}
\cM_\disc^{(n+1)}&=:
\dsp\int_\O\Big[\chi_\disc(\x)-\Pi_\disc P_\disc \bar p(t^{(n+1)})\Big]\\
\quad\quad\quad&\times\Big[f(\bar p^{(n+1)},\bar q^{(n+1)})+\div(\A(\x)\nabla\bar p^{(n+1)}(\x))-\dr_t\bar p^{(n+1)}(\x)\Big] \ud \x.
\end{aligned}
\end{equation}
Here $\cR_\disc^0$ and $\widetilde \cR_\disc^0$ are defined by \eqref{eq-err-inter-0} and $C$ is a positive constant depending only on $T$, $M$, $d_1$, $d_2$, $C_\disc$, and on small positive parameters linked to the Young's inequality.
\end{theorem}

\begin{proof}
By applying the interpolant $P_\disc$ to $\varphi=\bar p(t^{(n+1)})$ and the interpolant $\widetilde P_\disc$ to $\varphi=\bar q(t^{(n+1)})$, we obtain the following bounds, thanks to \eqref{new-eq-proof-SD-u} and \eqref{new-eq-proof-100} and to the Lipschitz continuity of $\nabla\bar p,\nabla\bar q: [0,T] \to L^2(\O)^d$ ensured by the hypothesis imposed on the solution.
\begin{equation}\label{eq-proof-1-u}
\begin{aligned}
&\| \nabla\bar p^{(n+1)}-\nabla_\disc  P_\disc \bar p(t^{(n+1)}) \|_{L^2(\O)^d}\\
&\leq \| \nabla\bar p^{(n+1)}-\nabla\bar p(t^{(n+1)}) \|_{L^2(\O)^d}
+\| \nabla\bar p(t^{(n+1)})-\nabla_\disc  P_\disc \bar p(t^{(n+1)}) \|_{L^2(\O)^d}\\
&\leq \delta t_\disc+\cS(\bar p(t^{(n+1)})),\mbox{ and }
\end{aligned}
\end{equation}
\begin{equation}\label{eq-proof-1-v}
\begin{aligned}
&\| \nabla\bar q^{(n+1)}-\nabla_\disc  \widetilde P_\disc \bar q(t^{(n+1)}) \|_{L^2(\O)^d}\\
&\leq \| \nabla\bar q^{(n+1)}-\nabla\bar q(t^{(n+1)}) \|_{L^2(\O)^d}
+\| \nabla\bar q(t^{(n+1)})-\nabla_\disc  \widetilde P_\disc \bar q(t^{(n+1)}) \|_{L^2(\O)^d}\\
&\leq \delta t_\disc+\widetilde\cS(\bar q(t^{(n+1)})).
\end{aligned}
\end{equation}

Since $\dr_t \bar p^{(n+1)},\dr_t \bar q^{(n+1)} \in H^2(\O)$, we can apply the interpolant $P_\disc$ to $\varphi:=\dr_t\bar p^{(n+1)}=\frac{\bar p(t^{(n+1)})-\bar p(t^{(n)})}{\delta t^{(n+\frac{1}{2})}}$, and the interpolant $\widetilde P_\disc$ to $\varphi:=\dr_t\bar q^{(n+1)}=\frac{\bar q(t^{(n+1)})-\bar q(t^{(n)})}{\delta t^{(n+\frac{1}{2})}}$. Note that thanks to the linearity of $P_\disc$ and $\widetilde P_\disc$, and \eqref{new-eq-proof-SD-u} and \eqref{new-eq-proof-100}, the following holds 
\begin{equation}\label{eq-proof-2-u}
\begin{aligned}
\norm{\frac{\Pi_\disc  P_\disc\bar p(t^{(n+1)})-\Pi_\disc  P_\disc \bar p(t^{(n)})}
{\delta t^{(n+\frac{1}{2})}}
-\partial_t \bar p^{(n+1)}}_{L^2(\O)}
&\leq \cS( \partial_t \bar p^{(n+1)}),\mbox{ and }
\end{aligned}
\end{equation}
\begin{equation}\label{eq-proof-2-v}
\begin{aligned}
\Big\| \frac{\Pi_\disc \widetilde P_\disc\bar q(t^{(n+1)})-\Pi_\disc  \widetilde P_\disc \bar q(t^{(n)})}
{\delta t^{(n+\frac{1}{2})}}
-\partial_t \bar q^{(n+1)} \Big\|_{L^2(\O)}
&\leq \cS( \partial_t \bar q^{(n+1)}).
\end{aligned}
\end{equation}

A direct application of the limit--conformity function \eqref{long-rm} to $\bpsi=:\A\nabla\bar p^{(n+1)} \in H_{\div}(\O)$ and to $\bpsi=:\B\nabla\bar q^{(n+1)}\in H_{\div}(\O)$ yields
\begin{equation}\label{eq-limi-conf-u}
\begin{aligned}
\dsp\int_\O \Big[\Pi_\disc w(\x)&\div( \A\nabla\bar p^{(n+1)}(\x))
+\A\nabla\bar p^{(n+1)}(\x)\cdot \nabla_\disc w(\x) \Big] \ud \x \\
&\leq W_\disc(\A\nabla\bar p^{(n+1)}) \| \nabla_\disc w \|_{L^2(\O)^d},\quad \forall w\in X_{\disc,0}, \mbox{ and }
\end{aligned}
\end{equation}
\begin{equation}\label{eq-limi-conf-v}
\begin{aligned}
\int_\O \Big[\Pi_\disc w(\x)&\div( \B\nabla\bar q^{(n+1)}(\x))
+\B\nabla\bar q^{(n+1)}(\x)\cdot \nabla_\disc w(\x) \Big] \ud \x \\
&\leq W_\disc(\B\nabla\bar q^{(n+1)}) \| \nabla_\disc w \|_{L^2(\O)^d},\quad \forall w\in X_{\disc,0}.
\end{aligned}
\end{equation}

Let us focus on each of the above inequalities separately. First, we note that the relations \eqref{pvi-obs1}--\eqref{pvi-obs3} are valid almost everywhere in $\O_T$ due to the conditions imposed on $\bar p$. Taking the average over time in $(t^{(n)},t^{(n+1)})$ ensures that $\dr_t \bar p^{(n+1)}-f(\bar p^{(n+1)},\bar q^{(n+1)}) \geq \div(\nabla\bar p^{(n+1)})$. Using the fact that $p^{(n+1)} \in \cK_\disc$, we obtain
\begin{equation}
\dsp\int_\O \Big(\Pi_\disc p^{(n+1)}-\chi_\disc\Big)\Big(f(\bar p^{(n+1)},\bar q^{(n+1)})+\div(\nabla\bar p^{(n+1)})-\dr_t\bar p^{(n+1)}\Big) \ud x \leq 0,
\end{equation}
which yields, for any $\varphi\in \cK_\disc$,
\begin{equation}\label{eq-diff-obs}
\begin{aligned}
&\int_\O \Big[\Pi_\disc p^{(n+1)}(\x)-\Pi_\disc \varphi(\x)\Big]\div(\nabla\bar p^{(n+1)}(\x)) \ud \x\\
&\leq\int_\O \Big[\chi_\disc(\x)-\Pi_\disc \varphi(\x)\Big]
\Big[f(\bar p^{(n+1)},\bar q^{(n+1)})+\div(\nabla\bar p^{(n+1)}(\x))-\dr_t\bar p^{(n+1)}(\x)\Big] \ud \x\\
&\quad-\int_\O \Big[\Pi_\disc p^{(n+1)}(\x)-\Pi_\disc \varphi(\x)\Big]\Big[f(\bar p^{(n+1)},\bar q^{(n+1)})-\dr_t\bar p^{(n+1)}(\x)\Big] \ud \x.
\end{aligned}
\end{equation}
Introduce the quantities $\chi$ and $\bar p^{(n+1)}$ in the first term of the RHS to attain
\begin{align*}
\int_\O &\Big[\Pi_\disc p^{(n+1)}(\x)-\Pi_\disc \varphi(\x)\Big]\div(\nabla\bar p^{(n+1)}(\x)) \ud \x\\
&\leq \int_\O \Big[\chi_\disc(\x)-\chi(\x)\Big]
\Big[f(\bar p^{(n+1)},\bar q^{(n+1)})+\div(\nabla\bar p^{(n+1)}(\x))-\dr_t\bar p^{(n+1)}(\x)\Big] \ud \x\\
&+\int_\O \Big[\chi(\x)-\bar p^{(n+1)}(\x)\Big]\Big[f(\bar p^{(n+1)},\bar q^{(n+1)})+\div(\nabla\bar p^{(n+1)}(\x))-\dr_t\bar p^{(n+1)}(\x)\Big] \ud \x
\\
&+\int_\O \Big[\bar p^{(n+1)}(\x)-\Pi_\disc \varphi(\x)\Big]\Big[f(\bar p^{(n+1)},\bar q^{(n+1)})+\div(\nabla\bar p^{(n+1)}(\x))-\dr_t\bar p^{(n+1)}(\x)\Big] \ud \x\\
&-\int_\O \Big[\Pi_\disc p^{(n+1)}(\x)- \Pi_\disc \varphi(\x) \Big]\Big[f(\bar p^{(n+1)},\bar q^{(n+1)})-\dr_t\bar p^{(n+1)}(\x)\Big] \ud \x.
\end{align*}
We note that the second term of the RHS vanishes owing to \eqref{pvi-obs1}. From the above inequality, we have, for all $\varphi\in \cK_\disc$,
\begin{align*}
\int_\O &\Big[\Pi_\disc \varphi(\x)-\Pi_\disc p^{(n+1)}(\x)\Big]\div(\nabla\bar p^{(n+1)}(\x)) \ud \x\\
&\geq
\int_\O \Big[\Pi_\disc \varphi(\x)-\Pi_\disc p^{(n+1)}(\x) \Big]\Big[\dr_t\bar p^{(n+1)}(\x)-f(\bar p^{(n+1)},\bar q^{(n+1)})\Big] \ud \x\\
&+\dsp\int_\O \Big[\Pi_\disc \varphi(\x)-\bar p^{(n+1)}(\x)\Big]\Big[f(\bar p^{(n+1)},\bar q^{(n+1)})+\div(\nabla\bar p^{(n+1)})(\x)-\dr_t\bar p^{(n+1)}(\x)\Big] \ud \x\\
&+\int_\O \Big[\chi(\x)-\chi_\disc(\x)\Big]\Big[f(\bar p^{(n+1)},\bar q^{(n+1)})+\div(\nabla\bar p^{(n+1)}(\x))-\dr_t\bar p^{(n+1)}(\x)\Big] \ud \x,
\end{align*}
which gives, for all $\varphi\in \cK_\disc$, thanks again to \eqref{pvi-obs1} 
\begin{align*}
\int_\O &\Big[\Pi_\disc \varphi(\x)-\Pi_\disc p^{(n+1)}(\x)\Big]\div(\nabla\bar p^{(n+1)}(\x)) \ud \x\\
&\geq
\int_\O \Big[\Pi_\disc \varphi(\x)-\Pi_\disc p^{(n+1)}(\x) \Big]\Big[\dr_t\bar p^{(n+1)}(\x)-f(\bar p^{(n+1)},\bar q^{(n+1)})\Big] \ud \x\\
&+\dsp\int_\O\Big[\Pi_\disc \varphi(\x)-\chi_\disc(\x)\Big]\Big[f(\bar p^{(n+1)},\bar q^{(n+1)})+\div(\nabla\bar p^{(n+1)})(\x)-\dr_t\bar p^{(n+1)}(\x) \Big] \ud \x.
\end{align*}
Apply the inequality \eqref{eq-limi-conf-u} to $w:=\varphi-p^{(n+1)}\in X_{\disc,0}$, and use the above relation with taking $\varphi:=P_\disc\bar p(t^{(n+1)}) \in \cK_\disc$ to arrive at
\begin{align*}
\int_\O &\Big[\Pi_\disc P_\disc\bar p(t^{(n+1)})-\Pi_\disc p^{(n+1)}(\x)\Big]\Big[\dr_t\bar p^{(n+1)}(\x)-f(\bar p^{(n+1)},\bar q^{(n+1)})\Big] \ud \x\\
&+\int_\O \Big[\nabla_\disc P_\disc\bar p(t^{(n+1)})-\nabla_\disc p^{(n+1)}(\x) \Big]\cdot \nabla\bar p^{(n+1)}(\x) \ud \x\\
&\leq W_\disc(\A\nabla\bar p^{(n+1)})\| \nabla_\disc (P_\disc\bar p(t^{(n+1)})-p^{(n+1)}) \|_{L^2(\O)^d}+\cM_\disc^{(n+1)},
\end{align*}
where $\cM_\disc^{(n+1)}$ is defined by \eqref{eq-M}. Inserting the term $f(\Pi_\disc p^{(n+1)},\Pi_\disc q^{(n+1)})$ in the above inequality yields
\begin{align*}
\int_\O &\Big[\Pi_\disc P_\disc\bar p(t^{(n+1)})-\Pi_\disc p^{(n+1)}(\x)\Big]\Big[\dr_t\bar p^{(n+1)}(\x)-f(\Pi_\disc p^{(n+1)},\Pi_\disc q^{(n+1)})\Big] \ud \x\\
&+\int_\O \Big[\Pi_\disc P_\disc\bar p(t^{(n+1)})-\Pi_\disc p^{(n+1)}(\x)\Big]
\Big[f(\Pi_\disc p^{(n+1)},\Pi_\disc q^{(n+1)})-f(\bar p^{(n+1)},\bar q^{(n+1)})\Big] \ud \x\\
&+\int_\O \Big[\nabla_\disc P_\disc\bar p(t^{(n+1)})-\nabla_\disc p^{(n+1)}(\x) \Big]\cdot \nabla\bar p^{(n+1)}(\x) \ud \x\\
&\leq W_\disc(\A\nabla\bar p^{(n+1)})\| \nabla_\disc (P_\disc\bar p(t^{(n+1)})-p^{(n+1)}) \|_{L^2(\O)^d}+\cM_\disc^{(n+1)}.
\end{align*}
Since $p$ is the solution to the approximate scheme \eqref{gs-pvi}, it follows that
\begin{equation}\label{eq-proof3}
\begin{aligned}
\dsp\int_\O &\Big[\Pi_\disc P_\disc\bar p(t^{(n+1)})-\Pi_\disc p^{(n+1)}(\x) \Big]\Big[\partial_t\bar p^{(n+1)}(\x)-\delta_\disc^{(n+\frac{1}{2})}p(\x)\Big] \ud \x\\
&\qquad+\int_\O \Big[\nabla_\disc P_\disc\bar p(t^{(n+1)})- \nabla_\disc p^{(n+1)}(\x) \Big]\cdot \Big[\nabla\bar p^{(n+1)}(\x)-\nabla_\disc p^{(n+1)}(\x)\Big] \ud \x\\
&\leq W_\disc(\A\nabla\bar p^{(n+1)})\| \nabla_\disc (P_\disc\bar p(t^{(n+1)})-p^{(n+1)}) \|_{L^2(\O)^d}\\
&+\int_\O \Big[\Pi_\disc \varphi(\x)-\Pi_\disc p^{(n+1)}(\x)\Big]\Big[f(\bar p^{(n+1)},\bar q^{(n+1)})-f(\Pi_\disc p^{(n+1)},\Pi_\disc q^{(n+1)})\Big] \ud \x\\
&+\cM_\disc^{(n+1)}.
\end{aligned}
\end{equation}
Combining \eqref{eq-proof3}, \eqref{eq-proof-2-u}, and \eqref{eq-proof-1-u}, and using the notation $\cR^{(k)}$ given by \eqref{eq-E}, one gets
\begin{equation}\label{eq-proof-6-u}
\begin{aligned}
\int_\O &\Pi_\disc (P_\disc\bar p(t^{(n+1)})-p^{(n+1)})\delta_\disc^{(n+\frac{1}{2})}\cR(\x) \ud \x\\
&\quad+\int_\O \nabla_\disc (P_\disc\bar p(t^{(n+1)})-p^{(n+1)})\cdot \A\nabla_\disc \cR^{(n+1)}(\x) \ud \x\\
&\leq\int_\O \Pi_\disc (P_\disc\bar p(t^{(n+1)})-p^{(n+1)})\Big[f(\bar p^{(n+1)},\bar q^{(n+1)})-f(\Pi_\disc p^{(n+1)},\Pi_\disc q^{(n+1)})\Big] \ud \x\\
&\quad+\Big[\delta t_\disc+\cS(\bar p(t^{(n+1)}))
+\cS(\dr_t\bar p^{(n+1)})
+W_\disc(\A\nabla\bar p^{(n+1)})\Big]\\
&\quad \times\| \nabla_\disc (P_\disc\bar p(t^{(n+1)})-p^{(n+1)}) \|_{L^2(\O)^d}
+\cM_\disc^{(n+1)},
\end{aligned}
\end{equation}
where
\[
\begin{aligned}
\delta_\disc^{(n+\frac{1}{2})}\cR:&=\frac{ \Pi_\disc \cR^{(n+1)}-\Pi_\disc \cR^{(n)} }{ \delta t^{(n+\frac{1}{2})} }\\
&=\Big[  \dsp\frac{ \Pi_\disc  P_\disc\bar p(t^{(n+1)})-\Pi_\disc  P_\disc(\bar p(t^{(n)})) }{ \delta t^{(n+\frac{1}{2})} }-\dr_t\bar p^{(n+1)} \Big]+\Big[ \dr_t\bar p^{(n+1)}-\delta_\disc^{(n+\frac{1}{2})} p \Big],\mbox{ and}
\end{aligned}
\]
\[
\A\nabla_\disc \cR^{(n+1)}= \A\left[ \nabla_\disc  (P_\disc\bar p(t^{(n+1)}))-\nabla\bar p^{(n+1)}  \right] + \A\left[\nabla\bar p^{(n+1)}-\nabla_\disc p^{(n+1)}  \right].
\]

Secondly, the adequate regularity of the solution $\bar q$ ensures that the strong formulation \eqref{pvi-obs4} is satisfied almost everywhere in the domain $\O_T$. Taking the average over time, in which $t^{(n)} \leq t \leq t^{(n+1)}$, results in $\dr_t \bar q^{(n+1)}-g(\bar p^{(n+1)},\bar q^{(n+1)}) = \div(\B\nabla\bar q^{(n+1)})$. Consequently, injecting this relation in \eqref{eq-limi-conf-v} yields 
\begin{equation}\label{eq-proof-3-v}
\begin{aligned}
\int_\O \Pi_\disc w(\x) \Big[ \dr_t\bar q^{(n+1)}(\x) &- g(\bar p^{(n+1)},\bar q^{(n+1)}) \Big]\ud \x
+\dsp\int_\O\B\nabla\bar q^{(n+1)}(\x)\cdot \nabla_\disc w(\x)  \ud \x \\
&\leq W_\disc(\B\nabla\bar q^{(n+1)}) \| \nabla_\disc w \|_{L^2(\O)^d}\quad \forall w\in X_{\disc,0}.
\end{aligned}
\end{equation}
By inserting the non-linear discrete term $g(\Pi_\disc p^{(n+1)},\Pi_\disc q^{(n+1)})$ in the above inequality, we obtain
\begin{equation}\label{eq-proof-4-v}
\begin{aligned}
\int_\O &\Pi_\disc w(\x) \Big[ \dr_t\bar q^{(n+1)}(\x)- g(\Pi_\disc p^{(n+1)},\Pi_\disc q^{(n+1)}) \Big] \ud \x\\
&\quad+\int_\O \Pi_\disc w(\x)\Big[ g(\Pi_\disc p^{(n+1)},\Pi_\disc q^{(n+1)}) - g(\bar p^{(n+1)},\bar q^{(n+1)}) \Big] \ud \x\\
&\quad+\int_\O\B\nabla\bar q^{(n+1)}(\x)\cdot \nabla_\disc w(\x)  \ud \x \\
&\leq W_\disc(\B\nabla\bar q^{(n+1)}) \| \nabla_\disc w \|_{L^2(\O)^d},\quad \forall w\in X_{\disc,0},
\end{aligned}
\end{equation}
which leads to, because $q$ satisfies the equality  \eqref{gs-pvi-obs2} in the discrete problem
\begin{equation}\label{eq-proof-5-v}
\begin{aligned}
\int_\O &\Pi_\disc w(\x) \Big[ \dr_t\bar q^{(n+1)}(\x)- \delta_\disc^{(n+\frac{1}{2})}q(\x) \Big] \ud \x\\
&\quad+\int_\O \Big[\nabla\bar q^{(n+1)}(\x)-\nabla_\disc q^{(n+1)} \Big]\cdot \nabla_\disc w(\x)  \ud \x \\
&\leq\int_\O \Pi_\disc w(\x)\Big[g(\bar p^{(n+1)},\bar q^{(n+1)})-g(\Pi_\disc p^{(n+1)},\Pi_\disc q^{(n+1)}) \Big] \ud \x\\
&\quad+W_\disc(\B\nabla\bar q^{(n+1)}) \| \nabla_\disc w \|_{L^2(\O)^d}\quad \forall w\in X_{\disc,0}.
\end{aligned}
\end{equation}
Combining \eqref{eq-proof-5-v}, \eqref{eq-proof-2-v}, and \eqref{eq-proof-1-v}, and using the notation $\widetilde \cR^{(k)}$ defined by \eqref{eq-E}, we derive, for all $w\in X_{\disc,0}$,
\begin{equation}\label{eq-proof-6-v}
\begin{aligned}
\int_\O &\Pi_\disc w(\x)\delta_\disc^{(n+\frac{1}{2})}\widetilde \cR(\x) \ud \x
+\int_\O \nabla_\disc w(\x)\cdot \B\nabla_\disc \widetilde \cR^{(n+1)}(\x) \ud \x\\
&\quad\leq\int_\O \Pi_\disc w(\x)\Big[g(\bar p^{(n+1)},\bar q^{(n+1)})-g(\Pi_\disc p^{(n+1)},\Pi_\disc q^{(n+1)}) \Big] \ud \x\\
&\quad\Big[\delta t_\disc+\widetilde\cS(\bar q(t^{(n+1)}))
+\widetilde\cS(\dr_t\bar q^{(n+1)})
+W_\disc(\B\nabla\bar q^{(n+1)})\Big]\|\nabla_\disc w\|_{L^2(\O)^d},
\end{aligned}
\end{equation}
where
\[
\begin{aligned}
\delta_\disc^{(n+\frac{1}{2})}\widetilde \cR:&=\frac{ \Pi_\disc\widetilde \cR^{(n+1)}-\Pi_\disc \widetilde \cR^{(n)} }{ \delta t^{(n+\frac{1}{2})} }\\
&=\Big[  \dsp\frac{ \Pi_\disc  \widetilde P_\disc\bar q(t^{(n+1)})-\Pi_\disc  \widetilde P_\disc(\bar q(t^{(n)})) }{ \delta t^{(n+\frac{1}{2})} }-\dr_t\bar q^{(n+1)} \Big]+\Big[ \dr_t\bar q^{(n+1)}-\delta_\disc^{(n+\frac{1}{2})} q \Big],
\end{aligned}
\]
\[
\B\nabla_\disc \widetilde \cR^{(n+1)}= \B\left[ \nabla_\disc  (\widetilde P_\disc\bar q(t^{(n+1)}))-\nabla\bar q^{(n+1)}  \right] + \B\left[\nabla\bar q^{(n+1)}-\nabla_\disc q^{(n+1)}  \right].
\]
Multiply \eqref{eq-proof-6-u} by $\delta t^{(n+\frac{1}{2})}$ and let $w:=\delta t^{(n+\frac{1}{2})}\widetilde \cR^{(n+1)}\in X_{\disc,0}$ in \eqref{eq-proof-6-v}. For both inequalities, we can sum over $n = 0,...,m-1$, for some $m\in \{1,...,N\}$ to obtain
\begin{equation}\label{eq-proof-7-u}
\begin{aligned}
&\dsp\sum_{n=0}^{m-1} \int_\O \Pi_\disc \cR^{(n+1)}(\x)\Big[\Pi_\disc \cR^{(n+1)}(\x)-\Pi_\disc \cR^{(n)}(\x)\Big] \ud \x\\
&\quad+d_1\dsp\sum_{n=0}^{m-1} \delta t^{(n+\frac{1}{2})} \| \nabla_\disc \cR^{(n+1)} \|_{L^2(\O)^d}^2\\
&\leq\dsp\sum_{n=0}^{m-1} \delta t^{(n+\frac{1}{2})}\int_\O \Pi_\disc \cR^{(n+1)}(\x)\Big[f(\bar p^{(n+1)},\bar q^{(n+1)})-f(\Pi_\disc p^{(n+1)},\Pi_\disc q^{(n+1)}) \Big] \ud \x\\
&\quad+\dsp\sum_{n=0}^{m-1} \delta t^{(n+\frac{1}{2})}\Big[\delta t_\disc+\cS(\bar p(t^{(n+1)}))
+\cS(\dr_t\bar p^{(n+1)})
+W_\disc(\A\nabla\bar p^{(n+1)})\Big]\| \nabla_\disc \cR^{(n+1)} \|_{L^2(\O)^d}\\
&\quad+\dsp\sum_{n=0}^{m-1} \delta t^{(n+\frac{1}{2})}\cM_\disc^{(n+1)},\mbox{ and }
\end{aligned}
\end{equation}
\begin{equation}\label{eq-proof-7-v}
\begin{aligned}
&\dsp\sum_{n=0}^{m-1} \int_\O \Pi_\disc \widetilde \cR^{(n+1)}(\x)\Big[\Pi_\disc \widetilde \cR^{(n+1)}(\x)-\Pi_\disc \widetilde \cR^{(n)}(\x)\Big] \ud \x\\
&+d_1\dsp\sum_{n=0}^{m-1} \delta t^{(n+\frac{1}{2})} \| \nabla_\disc \widetilde \cR^{(n+1)} \|_{L^2(\O)^d}^2\\
&\leq\dsp\sum_{n=0}^{m-1} \delta t^{(n+\frac{1}{2})}\int_\O \Pi_\disc \widetilde \cR^{(n+1)}(\x)[g(\bar p^{(n+1)},\bar q^{(n+1)})-g(\Pi_\disc p^{(n+1)},\Pi_\disc q^{(n+1)}) \Big] \ud \x\\
&\quad+\dsp\sum_{n=0}^{m-1} \delta t^{(n+\frac{1}{2})}\Big[\delta t_\disc+\widetilde\cS(\bar q(t^{(n+1)}))
+\widetilde\cS(\dr_t\bar q^{(n+1)})
+W_\disc(\B\nabla\bar q^{(n+1)})]\| \nabla_\disc \widetilde \cR^{(n+1)} \|_{L^2(\O)^d}.
\end{aligned}
\end{equation}
Now, we apply the formula, $r(r-s)\geq \frac{1}{2}r^2-\frac{1}{2}s^2,\; \forall r,s \in \RR$, to the first term in the LHS in both inequalities, and perform the Young's inequality with small parameters $\varepsilon_7,\varepsilon_8>0$ for the second terms in the RHS of both inequalities. It follows that
\begin{equation}\label{eq-proof-8-u}
\begin{aligned}
&\frac{1}{2}\dsp\int_\O(\Pi_\disc \cR^{(m)}(\x))^2 \ud \x
+d_1\dsp\sum_{n=0}^{m-1} \delta t^{(n+\frac{1}{2})} \| \nabla_\disc \cR^{(n+1)} \|_{L^2(\O)^d}^2 \\
&\leq \frac{1}{2}\dsp\int_\O (\Pi_\disc \cR^{(0)}(\x))^2 \ud \x
+\frac{\varepsilon_7}{2}\dsp\sum_{n=0}^{m-1} \delta t^{(n+\frac{1}{2})} \| \nabla_\disc \cR^{(n+1)} \|_{L^2(\O)^d}^2\\
&\quad+\dsp\sum_{n=0}^{m-1} \delta t^{(n+\frac{1}{2})}\int_\O \Pi_\disc \cR^{(n+1)}(\x)\Big[f(\bar p^{(n+1)},\bar q^{(n+1)})-f(\Pi_\disc p^{(n+1)},\Pi_\disc q^{(n+1)}) \Big] \ud \x\\
&\quad+\frac{1}{2\varepsilon_7}\dsp\sum_{n=0}^{m-1} \delta t^{(n+\frac{1}{2})} \Big[\delta t_\disc+\cS(\bar p(t^{(n+1)}))
+\cS(\dr_t\bar p^{(n+1)})
+W_\disc(\A\nabla\bar p^{(n+1)})\Big]^2\\
&\quad+\dsp\sum_{n=0}^{m-1} \delta t^{(n+\frac{1}{2})}\cM_\disc^{(n+1)},\mbox{ and }
\end{aligned}
\end{equation}
\begin{equation}\label{eq-proof-8-v}
\begin{aligned}
&\frac{1}{2}\dsp\int_\O(\Pi_\disc \widetilde \cR^{(m)}(\x))^2 \ud \x
+d_1\dsp\sum_{n=0}^{m-1} \delta t^{(n+\frac{1}{2})} \| \nabla_\disc \widetilde \cR^{(n+1)} \|_{L^2(\O)^d}^2 \\
&\leq \frac{1}{2}\dsp\int_\O (\Pi_\disc \widetilde \cR^{(0)}(\x))^2 \ud \x
+\frac{\varepsilon_8}{2}\dsp\sum_{n=0}^{m-1} \delta t^{(n+\frac{1}{2})} \| \nabla_\disc \widetilde \cR^{(n+1)} \|_{L^2(\O)^d}^2\\
&\quad+\dsp\sum_{n=0}^{m-1} \delta t^{(n+\frac{1}{2})}\int_\O \Pi_\disc \widetilde \cR^{(n+1)}(\x)\Big[g(\bar p^{(n+1)},\bar q^{(n+1)})-g(\Pi_\disc p^{(n+1)},\Pi_\disc q^{(n+1)}) \Big] \ud \x\\
&\quad+\frac{1}{2\varepsilon_8}\dsp\sum_{n=0}^{m-1} \delta t^{(n+\frac{1}{2})} \Big[\delta t_\disc+\widetilde\cS(\bar q(t^{(n+1)}))
+\widetilde\cS(\dr_t\bar q^{(n+1)})
+W_\disc(\B\nabla\bar q^{(n+1)})\Big]^2.
\end{aligned}
\end{equation}
The initial condition terms can be handled in the following manner.
\begin{equation}\label{eq-proof-9-u}
\begin{aligned}
\| \Pi_\disc \cR^{(0)} \|_{L^2(\O)} &\leq \| \Pi_\disc P_\disc \bar p(0)-\bar p(0) \|_{L^2(\O)}
+\| \bar p(0)-\Pi_\disc J_\disc\bar p(0) \|_{L^2(\O)}\\
&\leq \cS(\bar p(0)) + \cR_\disc^0,\mbox{ and }
\end{aligned}
\end{equation}
\begin{equation}\label{eq-proof-9-v}
\begin{aligned}
\| \Pi_\disc \widetilde \cR^{(0)} \|_{L^2(\O)} &\leq \| \Pi_\disc P_\disc \bar q(0)-\bar q(0) \|_{L^2(\O)}
+\| \bar q(0)-\Pi_\disc J_\disc\bar q(0) \|_{L^2(\O)}\\
&\leq \cS(\bar q(0)) + \widetilde \cR_\disc^0.
\end{aligned}
\end{equation}
After injecting \eqref{eq-proof-9-u} and \eqref{eq-proof-F-1} proved in Lemma \ref{lemma-reaction} in \eqref{eq-proof-8-u}, we arrive at, thanks to the fact that $\sum_{n=0}^{m-1}\delta t^{(n+\frac{1}{2})}\leq T$,
\begin{equation}\label{new-E-u}
\begin{aligned}
&\frac{1}{2}\dsp\int_\O (\Pi_\disc  \cR^{(m)}(\x))^2 \ud \x
+\frac{d_1-\varepsilon_7}{2}\dsp\sum_{n=0}^{m-1} \delta t^{(n+\frac{1}{2})} \| \nabla_\disc  \cR^{(n+1)} \|_{L^2(\O)^d}^2 \\
&\leq C_1\dsp\sum_{n=0}^{m-1} \delta t^{(n+\frac{1}{2})}\|\Pi_\disc  \cR^{(n+1)}\|_{L^2(\O)}^2
+C_2\dsp\sum_{n=0}^{m-1} \delta t^{(n+\frac{1}{2})} \|\Pi_\disc \widetilde \cR^{(n+1)}\|_{L^2(\O)}^2\\
&\quad+TC_3 \Big[\delta t_\disc+\cS(\bar p(t^{(n+1)}))\Big]^2
+TC_4[\delta t_\disc+\widetilde\cS(\bar q(t^{(n+1)})) \Big]^2\\
&\quad+T\Big[\delta t_\disc+\cS(\bar p(t^{(n+1)}))
+\cS(\dr_t\bar p^{(n+1)})
+W_\disc(\A\nabla\bar p^{(n+1)})\Big]^2\\
&\quad+\Big[\cS(\bar p(0)) + \cR_\disc^0\Big]^2
+\dsp\sum_{n=0}^{m-1} \delta t^{(n+\frac{1}{2})}\cM_\disc^{(n+1)}.
\end{aligned}
\end{equation}
Also, injecting \eqref{eq-proof-9-v} and \eqref{eq-proof-G-1} proved in Lemma \ref{lemma-reaction} in \eqref{eq-proof-8-u}, we obtain
\begin{equation}\label{new-E-v}
\begin{aligned}
&\frac{1}{2}\dsp\int_\O (\Pi_\disc \widetilde \cR^{(m)}(\x))^2 \ud \x
+\frac{d_1-\varepsilon_8}{2}\dsp\sum_{n=0}^{m-1} \delta t^{(n+\frac{1}{2})} \| \nabla_\disc \widetilde \cR^{(n+1)} \|_{L^2(\O)^d}^2 \\
&\leq C_6\dsp\sum_{n=0}^{m-1} \delta t^{(n+\frac{1}{2})}\|\Pi_\disc \widetilde \cR^{(n+1)}\|_{L^2(\O)}^2
+C_5\dsp\sum_{n=0}^{m-1} \delta t^{(n+\frac{1}{2})} \|\Pi_\disc \cR^{(n+1)}\|_{L^2(\O)}^2\\
&\quad+TC_7 \Big[\delta t_\disc+\cS(\bar p(t^{(n+1)}))\Big]^2
+TC_8\Big[\delta t_\disc+\widetilde\cS(\bar q(t^{(n+1)})) \Big]^2\\
&\quad+T\Big[\delta t_\disc+\widetilde\cS(\bar q(t^{(n+1)}))
+\widetilde\cS(\dr_t\bar q^{(n+1)})
+W_\disc(\B\nabla\bar q^{(n+1)})\Big]^2\\
&\quad+\Big[\widetilde\cS(\bar q(0)) + \widetilde \cR_\disc^0\Big]^2.
\end{aligned}
\end{equation}
Applying the discrete Gronwall’s Lemma \cite[Lemma 10.5]{thomee-2007} leads to
\begin{equation}\label{eq-proof-10-u}
\begin{aligned}
&\frac{1}{2}\dsp\int_\O(\Pi_\disc  \cR^{(m)}(\x))^2 \ud \x
+\frac{d_1-\varepsilon_7}{2}\dsp\sum_{n=0}^{m-1} \delta t^{(n+\frac{1}{2})} \| \nabla_\disc  \cR^{(n+1)} \|_{L^2(\O)^d}^2 \\
&\leq \exp(TC_1)\Big[C_2\dsp\sum_{n=0}^{m-1} \delta t^{(n+\frac{1}{2})} \|\Pi_\disc \widetilde \cR^{(n+1)}\|_{L^2(\O)}^2\\
&\quad+TC_3 \Big(\delta t_\disc+\cS(\bar p(t^{(n+1)}))\Big)^2
+TC_4\Big(\delta t_\disc+\widetilde\cS(\bar q(t^{(n+1)})) \Big)^2\\
&\quad+T\Big(\delta t_\disc+\cS(\bar p(t^{(n+1)}))
+\cS(\dr_t\bar p^{(n+1)})
+W_\disc(\A\nabla\bar p^{(n+1)})\Big)^2\\
&\quad+\Big(\cS(\bar p(0)) + \cR_\disc^0\Big)^2
+\dsp\sum_{n=0}^{m-1} \delta t^{(n+\frac{1}{2})}\cM_\disc^{(n+1)}
\Big]
,\mbox{ and }
\end{aligned}
\end{equation}
\begin{equation}\label{eq-proof-10-v}
\begin{aligned}
&\frac{1}{2}\dsp\int_\O(\Pi_\disc \widetilde \cR^{(m)}(\x))^2 \ud \x
+\frac{d_1-\varepsilon_8}{2}\dsp\sum_{n=0}^{m-1} \delta t^{(n+\frac{1}{2})} \| \nabla_\disc \widetilde \cR^{(n+1)} \|_{L^2(\O)^d}^2 \\
&\leq \exp(TC_6)\Big[C_5\dsp\sum_{n=0}^{m-1} \delta t^{(n+\frac{1}{2})} \|\Pi_\disc \cR^{(n+1)}\|_{L^2(\O)}^2\\
&\quad+TC_7 \Big(\delta t_\disc+\cS(\bar p(t^{(n+1)}))\Big)^2
+TC_8\Big(\delta t_\disc+\widetilde\cS(\bar q(t^{(n+1)})) \Big)^2\\
&\quad+T\Big(\delta t_\disc+\widetilde\cS(\bar q(t^{(n+1)}))
+\widetilde\cS(\dr_t\bar q^{(n+1)})
+W_\disc(\B\nabla\bar q^{(n+1)})\Big)^2\\
&\quad+\Big(\widetilde\cS(\bar q(0)) + \widetilde \cR_\disc^0\Big)^2
\Big].
\end{aligned}
\end{equation}
The simplification of the quantities on the RHSs with those on the LHSs allows us to sum both inequalities and to derive the following relations, thanks to the discrete Poincaré inequality (coming from \eqref{corc-eq})
\begin{equation}\label{eq-proof-12-1}
\begin{aligned}
&\Big[\frac{d_1-\varepsilon_1}{2}-C_\disc C_5\exp(TC_2)\Big]\dsp\sum_{n=0}^{m-1} \delta t^{(n+\frac{1}{2})} \| \nabla_\disc  \cR^{(n+1)} \|_{L^2(\O)^d}^2\\
&\quad+\Big[\frac{d_1-\varepsilon_2}{2}-C_\disc C_2\exp(TC_1)\Big]\dsp\sum_{n=0}^{m-1} \delta t^{(n+\frac{1}{2})} \| \nabla_\disc \widetilde \cR^{(n+1)} \|_{L^2(\O)^d}^2\\
&\leq
\mathbb I_3,
\mbox{ and }
\end{aligned}
\end{equation}
\begin{equation}\label{eq-proof-12-2}
\begin{aligned}
\frac{1}{2}\|\Pi_\disc  \cR^{(m)}\|_{L^2(\O)}^2
+\frac{1}{2}\|\Pi_\disc  \widetilde \cR^{(m)}\|_{L^2(\O)}^2
\leq \mathbb I_3,
\end{aligned}
\end{equation}
where
\[
\begin{aligned}
\mathbb I_3&=
C_9 \Big[ \Big(\delta t_\disc+\cS(\bar p(t^{(n+1)}))\Big)^2
+\Big(\delta t_\disc+\widetilde\cS(\bar q(t^{(n+1)})) \Big)^2\\
&\quad+\Big(\delta t_\disc+\widetilde\cS(\bar q(t^{(n+1)}))
+\widetilde\cS(\dr_t\bar q^{(n+1)})
+W_\disc(\B\nabla\bar q^{(n+1)})\Big)^2\\
&\quad+\Big(\delta t_\disc+\cS(\bar p(t^{(n+1)}))
+\cS(\dr_t\bar p^{(n+1)})
+W_\disc(\A\nabla\bar p^{(n+1)})\Big)^2\\
&\quad+(\widetilde\cS(\bar q(0)) + \widetilde \cR_\disc^0)^2
+(\cS(\bar p(0)) + \cR_\disc^0)^2
+\dsp\sum_{n=0}^{m-1} \delta t^{(n+\frac{1}{2})}\cM_\disc^{(n+1)}
\Big].
\end{aligned}
\]
Combining \eqref{new-eq-proof-SD-u}, \eqref{new-eq-proof-100}, and \eqref{eq-proof-12-2}, using the relation, $(r+s)^{1/2}\leq r^{1/2}+s^{1/2}$, for all $r,s\in \RR^{+}$, and utilising the triangle inequality, we obtain, for all $m\in\{1,...,N-1\}$,
\begin{equation}\label{eq-proof-11-u}
\begin{aligned}
&\|\Pi_\disc p^{(m)}-\bar p(t^{(m)})\|_{L^2(\O)}
+\|\Pi_\disc q^{(m)}-\bar q(t^{(m)})\|_{L^2(\O)}\\
&\leq \|\Pi_\disc \cR^{(m)}\|_{L^2(\O)}
+\|\Pi_\disc P_\disc\bar p(t^{(m)})-\bar p(t^{(m)})\|_{L^2(\O)}\\
&\quad+ \|\Pi_\disc \widetilde \cR^{(m)}\|_{L^2(\O)}
+\|\Pi_\disc P_\disc\bar q(t^{(m)})-\bar q(t^{(m)})\|_{L^2(\O)}\\
&\leq 
C \sqrt{\mathbb I_3}+\cS(\bar p(t^{(m)}))+\sqrt{2}\cS(\bar q(t^{(m)}))\\
&\leq C\Big[ \delta t_\disc+\cS(\bar p(t^{(m)}))
+\widetilde\cS(\bar q(t^{(m)}))
+\cS(\dr_t\bar p^{(m)})
+\widetilde\cS(\dr_t\bar q^{(m)})\\
&\quad+W_\disc(\A\nabla\bar p^{(m)})
+W_\disc(\B\nabla\bar q^{(m)})
+\widetilde\cS(\bar q(0)) + \widetilde \cR_\disc^0
+\cS(\bar p(0)) + \cR_\disc^0\\
&\quad+\Big(\dsp\sum_{n=0}^{N-1} \delta t^{(n+\frac{1}{2})}\cM_\disc^{(n+1)}\Big)^{1/2}
\Big].
\end{aligned}
\end{equation}
Together with the triangle inequality and the Lipschitz continuity of the solutions $\bar p,\bar q:[0,T]\to H^1(\O)$, we conclude the first desired estimate \eqref{eq-est-1-theorem}.

Also, Combining \eqref{new-eq-proof-SD-u}, \eqref{new-eq-proof-100}, and \eqref{eq-proof-12-1} with $m=N-1$, using the relation, $(r+s)^{1/2}\leq r^{1/2}+s^{1/2}$, for all $r,s\in \RR^{+}$, and using the triangle inequality, we have
\begin{equation}\label{eq-proof-13-u}
\begin{aligned}
&\dsp\sum_{n=0}^{N-1}\delta t^{(n+\frac{1}{2})}\|\nabla_\disc p^{(n+1)} - \nabla\bar p(t^{(n+1)})\|_{L^2(\O)^d}
+\dsp\sum_{n=0}^{N-1}\delta t^{(n+\frac{1}{2})}\|\nabla_\disc q^{(n+1)} - \nabla\bar q(t^{(n+1)})\|_{L^2(\O)^d}\\
&\leq \dsp\sum_{n=0}^{N-1}\delta t^{(n+\frac{1}{2})}\|\nabla_\disc \cR^{(n+1)}\|_{L^2(\O)^d}
+\dsp\sum_{n=0}^{N-1}\delta t^{(n+\frac{1}{2})}\|\nabla_\disc P_\disc\bar p(t^{(n+1)})-\nabla\bar p(t^{(n+1)})\|_{L^2(\O)^d}\\
&\quad+\dsp\sum_{n=0}^{N-1}\delta t^{(n+\frac{1}{2})}\|\nabla_\disc \widetilde \cR^{(n+1)}\|_{L^2(\O)^d}
+\dsp\sum_{n=0}^{N-1}\delta t^{(n+\frac{1}{2})}\|\nabla_\disc P_\disc\bar q(t^{(n+1)})-\nabla\bar q(t^{(n+1)})\|_{L^2(\O)^d}\\
&\leq \dsp\sum_{n=0}^{N-1}\delta t^{(n+\frac{1}{2})}\Big[C \sqrt{\mathbb I_3}+\cS(\bar p(t^{(m)}))+\sqrt{2}\cS(\bar q(t^{(m)}))\Big]\\
&\leq C\dsp\sum_{n=0}^{N-1}\delta t^{(n+\frac{1}{2})}
\Big[ \delta t_\disc+\cS(\bar p(t^{(n+1)}))
+\widetilde\cS(\bar q(t^{(n+1)}))
+\cS(\dr_t\bar p^{(n+1)})
+\widetilde\cS(\dr_t\bar q^{(n+1)})\\
&\quad+W_\disc(\A\nabla\bar p^{(n+1)})
+W_\disc(\B\nabla\bar q^{(n+1)})
+\widetilde\cS(\bar q(0)) + \widetilde \cR_\disc^0
+\cS(\bar p(0)) + \cR_\disc^0\\
&\quad+\Big(\cM_\disc^{(n+1)}\Big)^{1/2}
\Big].
\end{aligned}
\end{equation}
Together with the triangle inequality and the Lipschitz continuity of $\nabla\bar p,\nabla\bar q:[0,T]\to H^1(\O)$, we conclude the second desired estimate \eqref{eq-est-2-theorem}, which completes the proof.
\end{proof}

\begin{remark}\label{remark-1}
Theorem \ref{thm-err-rm} establishes the first general error estimates for the approximation of the model \eqref{pvi-obs}, applicable to both conforming and non-conforming methods. In contrast, the estimate presented in \cite{REF-1} is restricted to the $\mathbb P1$ finite elements method.

The regularity assumptions made on the solution in Theorem \ref{thm-err-rm} are required to obtain convergence rates. However, we establish in \cite{Alnashri-2022} the convergence of the gradient schemes \eqref{gs-pvi} under natural conditions stated in Assumptions \ref{assump-1}. The existence of a continuous solution to the problem \eqref{pvi-obs-weak} will follow from the convergence of the gradient schemes and Lemma \ref{lemma-1}.
\end{remark}

\begin{remark}[Convergence rates]
Let $h_\disc$ be the space size defined as in \cite[Definition 2.22]{B-10} with a slight modification, thanks to the continuously embedded spaces $W^{2,\infty}(\O)$ and $W^{1,\infty}(\O)^d$,
\[
\begin{aligned}
h_\disc:=\dsp\max\Big\{ 
\dsp\sup_{v\in W^{2,\infty}(\O)\setminus\{0\}}\frac{S(v)}{\|v\|_{W^{2,\infty}(\O)}},
\dsp\sup_{{\bf v}\in W^{1,\infty}(\O)^d\setminus\{0\}}\frac{W_\disc({\bf v})}{\|{\bf v}\|_{W^{1,\infty}(\O)^d}}
\Big\},
\end{aligned}
\]
and thus fulfils
\begin{subequations}\label{hd}
\begin{equation}\label{hd-1}
\forall v \in W^{2,\infty}(\O) \cap H_0^1(\O),\quad S(v)\leq h_\disc \|v\|_{W^{2,\infty}(\O)},
\end{equation}
\begin{equation}\label{hd-3}
\forall v \in W^{2,\infty}(\O) \cap H_0^1(\O),\quad \widetilde S(v)\leq h_\disc \|v\|_{W^{2,\infty}(\O)},
\end{equation}
\begin{equation}\label{hd-3}
\forall {\bf v} \in W^{1,\infty}(\O)^d \cap H_0^1(\O)^d,\quad W_\disc({\bf v})\leq h_\disc\|{\bf v}\|_{W^{1,\infty}(\O)^d}.
\end{equation}
\end{subequations}
Theorem \ref{thm-err-rm} provides convergence rates relative to both the time discretisation and the mesh size ($h_\mathcal M$)  is related to the space size $h_\disc$. From Relations \eqref{eq-def-S}, it follows that the convergence rates shown in Theorem \ref{thm-err-rm} depend on the parameters $S$, $W_\disc$, and $\cM_\disc$. Based on these relations, both terms $S$ and $W_\disc$ are of order $\mathcal O(h)$ for all common low-order gradient discretisations, employing optimal interpolants.

Let us now discuss the estimation of the term $\cM_\disc$. For a first-order conforming numerical method, we can ensure that $\|\chi-\chi_\disc\|_{L^2(\O)}=\mathcal O(h^2)$ if $\chi\in H^2(\O)$, and we can construct an interpolant $P_\disc$ satisfyning $\|\Pi_\disc P_\disc\bar p(t^{(n+1)})-\bar p(t^{(n+1)})\|_{L^2(\O)}=\mathcal O(h^2)$. It is thus possible to conclude that $\cM_\disc=\mathcal O(h^2)$ since can write, thanks to \eqref{pvi-obs1}
\[
\begin{aligned}
\cM_\disc^{(n+1)}&=
\dsp\int_\O\Big[ \chi_\disc(\x)-\chi(\x) \Big] \Big[ \bar p(t^{(n+1)})-\Pi_\disc P_\disc \bar p(t^{(n+1)}) \Big]\\
&\quad\quad\quad\times\Big[f(\bar p^{(n+1)},\bar q^{(n+1)})+\div(\A(\x)\nabla\bar p^{(n+1)}(\x))-\dr_t\bar p^{(n+1)}(\x)\Big] \ud \x.
\end{aligned}
\]

For non-conforming reconstructions, such as the HMM method used in our numerical simulation, it is possible to follow the proof of \cite[Theorem 2.13]{YJ-2016} (with letting $F:=f(\bar p^{(n+1)},\bar q^{(n+1)})+\div(\A(\x)\nabla\bar p^{(n+1)}(\x))-\dr_t\bar p^{(n+1)}(\x)$) to obtain the estimate $\cM_\disc=\mathcal O(h^2)$ under suitable assumptions that $\bar p -\chi \in W^{1,\infty}(0,T;W^{2,\infty}(\O))$.

\end{remark}

\section{Numerical tests}\label{sec-numerical}
This section presents numerical computations illustrating the performance of a specific scheme within the gradient discretisation framework, namely the mixed finite volume method. For a detailed formulation of the scheme applied to the model \eqref{pvi-obs}, we refer the reader to \cite[Section 5]{Alnashri-2022}.

\begin{test}
Let $\bf{Id}$ denote to the $2\times 2$ identity matrix, $\mathcal B({\bf 0},r)$ to an open disk, which is of center ${\bf 0}=(0,0)$ and radius $r$, and ${\bf 1}_A$ to a characteristic function of a set $A$. We 
consider \cite[Example 5.3]{REF-1}, where $\O=(-1,1)^2$, $\A=(0.01)\bf{Id}$, $\B=(0.5)\bf{Id}$, $\chi=0.3$, $p_0={\bf 1}_{\mathcal B({\bf 0},0.3)}$, $q_0={\bf 1}_{\mathcal B({\bf 0},0.75)}$, and the reaction functions are given by
\[
f(\bar p,\bar q)=\dsp\frac{5\bar q \bar p}{\bar q+0.7} \mbox{ and } g(\bar p,\bar q)=\dsp\frac{-0.5\bar q \bar p}{\bar q+0.7}.
\]

Figure \ref{fig-1} shows the graphs of the evolution of the approximate solutions $p$ and $q$ computed on meshes comprising 4443 hexagonal cells at different times. The evolution appears comparable to the one obtained with the linear finite elements method.
\begin{figure}[ht]
	\begin{center}
	\begin{tabular}{cc}
	\includegraphics[width=0.40\linewidth]{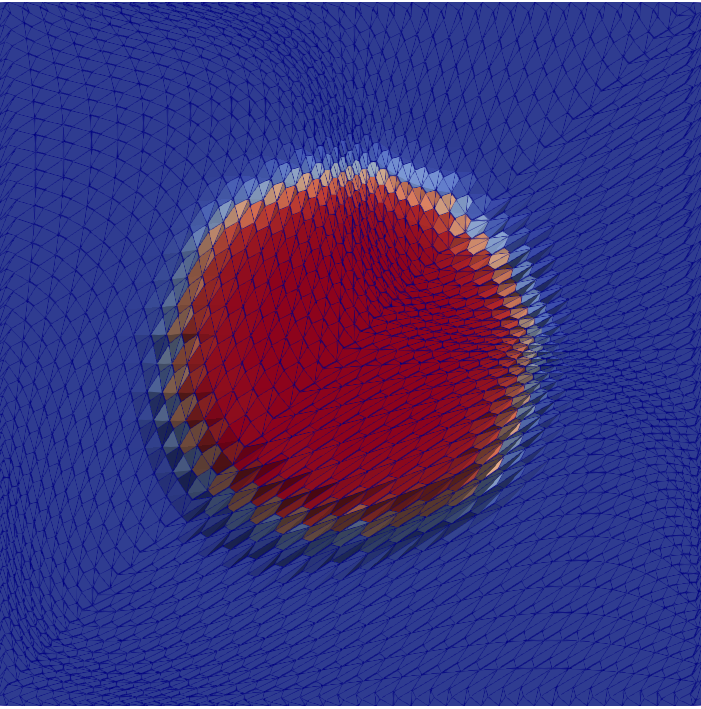} & \includegraphics[width=0.40\linewidth]{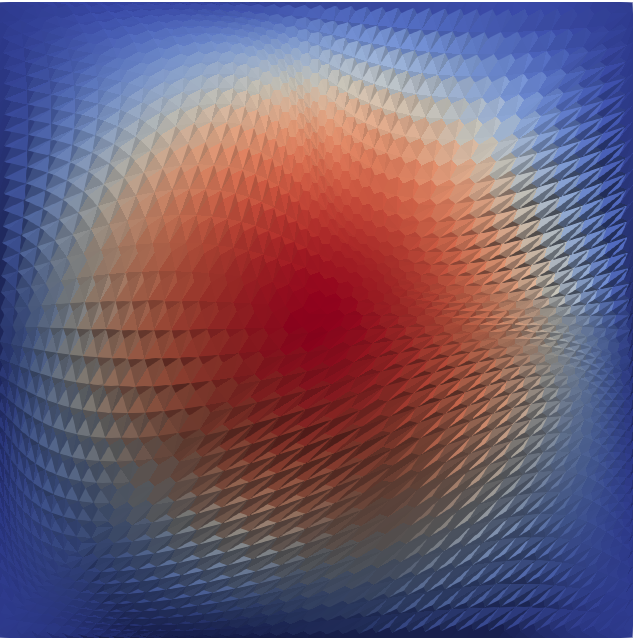}\\
	$T=0.1$ & $T=0.1$\\
	\includegraphics[width=0.40\linewidth]{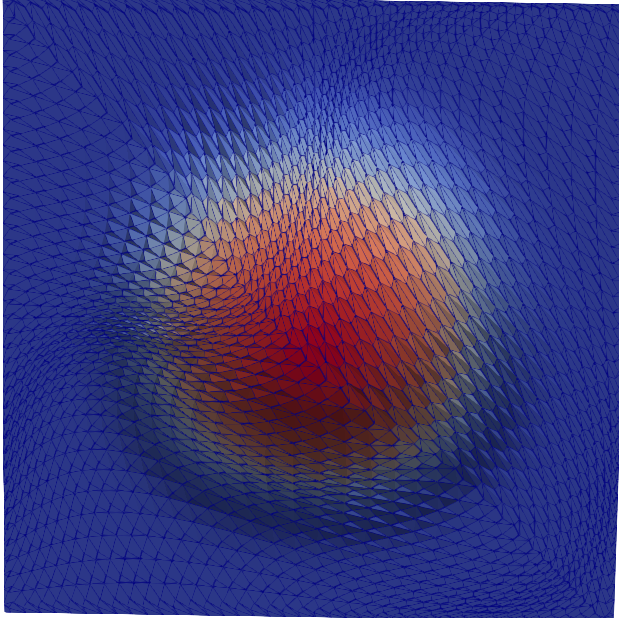} & \includegraphics[width=0.40\linewidth]{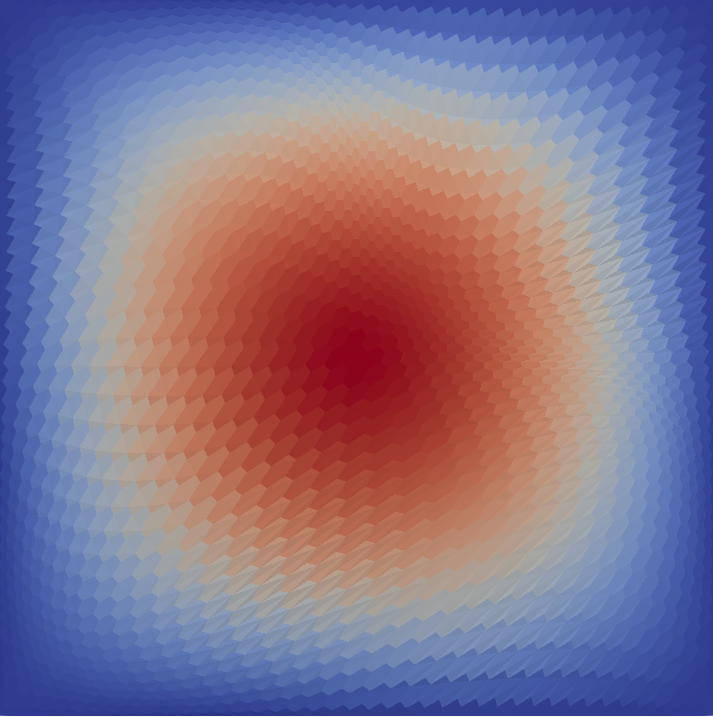}\\
	$T=1$ & $T=1$\\
	\includegraphics[width=0.40\linewidth]{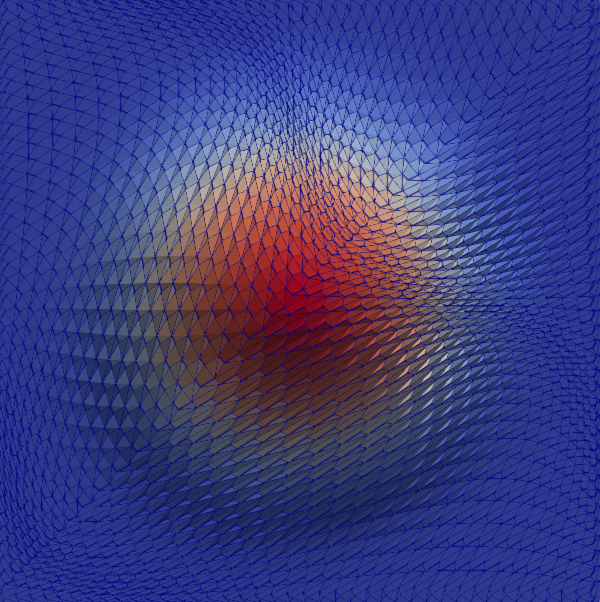} & \includegraphics[width=0.40\linewidth]{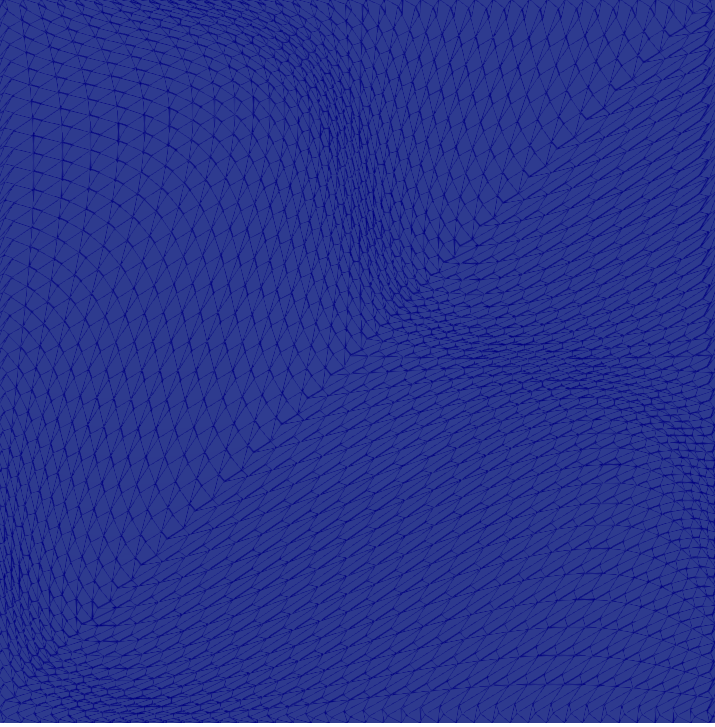}\\
	$T=2$ & $T=2$\\
	\end{tabular}
	\end{center}
	\caption{Evolution of $p$ (left column) and $q$ (right column).}
	\label{fig-1}
\end{figure}

\end{test}

\begin{test} To precisely evaluate convergence, we design a test case with an analytical solution, enabling exact measurement of the convergence orders. This contrasts with most existing studies, such as \cite{REF-1,REF-2}, that rely on cases without known analytical solutions and instead estimate convergence rates by comparing coarse-mesh approximations against a fine-mesh reference solution.

We begin with defining the time and space function $\alpha:\O\times[0,T]$, and the time dependent functions $\;\beta,\; \gamma,\; \zeta:[0,T] \to \RR^+$ by
\[
\alpha(\x,t)=\Big[ \Big(x-\frac{1}{3}\cos(4\pi t) \Big)^2 + \Big(y-\frac{1}{3}\sin(4\pi t) \Big)^2\Big]^{\frac{1}{2}}, \quad \beta(t)=\frac{1}{3}+0.3 \sin(16\pi t),
\]
\[
\gamma(t)=\frac{1}{3}\cos(4\pi t), \quad \mbox{ and }\quad \zeta(t)=\frac{1}{3}\sin(4\pi t).
\]

Letting $\O=(-1,1)^2$, $\O^+:=\{ \x\in \O\; : \; \alpha(\x,t) > \beta(\x,t) \} \mbox{ and } \O^{-}=\O-\O^+$ and the final time $T=0.25$, we consider the model \eqref{pvi-obs} with $\A={\bf Id}$, $\B=(0.25){\bf Id}$, and the barrier $\chi=0$. The reaction functions are given by
\[
f(\bar p,\bar q)=(\bar p+\bar q)^2 \quad \mbox{ and }\quad  g(\bar p, \bar q)=\bar p(1-\bar p\bar q).
\]
The exact solution is given by
\[
\bar p(\x,t)=\begin{cases}
\hfill \frac{1}{2} \Big( \alpha^2(\x,t) - \beta^2(t) \Big)^2&,\mbox{ if }\x \in \O^+,\\
0&,\mbox{ if }\x \in \O^-,
\end{cases}
\]
and
\[
\bar q(\x,t)=\exp(x+y+0.5 t).
\]

In this test case, the boundary conditions are non-homogeneous and derived directly from the analytical solution. Equation \eqref{pvi-obs1} includes a source term derived from the known exact solution, while Equation \eqref{pvi-obs4} is subject to the following source function.
\[
S(\x,t)=\begin{cases}
\hfill 4 \Big[ \beta^2(t)-2\alpha^2(\x,t)-\frac{1}{2}\left( \alpha^2(\x,t)-\beta^2(t) \right) \left(\gamma(t)+\beta(t)\frac{\ud \beta}{\ud t}(t) \right) \Big]&,\mbox{ if }\x \in \O^+,\\
4 \beta^2(t)\left[\alpha^2(\x,t)- \beta^2(t)-1 \right]&,\mbox{ if }\x \in \O^-.
\end{cases}
\]
Figure \ref{fig-2} presents the convergence graphs at the final time $T=2$. As predicted by Theorem \ref{thm-err-rm}, the results confirm a first-order convergence rate with respect to the mesh size $h$.

\begin{figure}[ht]
	\begin{center}
	\includegraphics[scale=0.4]{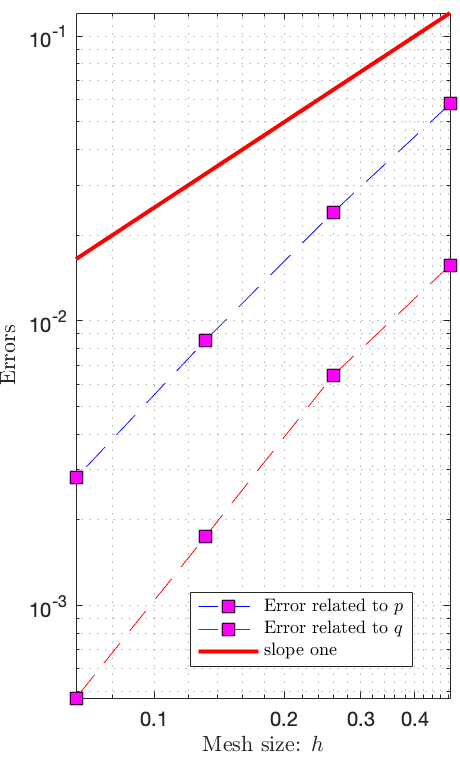}
	\end{center}
\caption{The errors on hexagonal meshes.}	
\label{fig-2}
\end{figure}

\end{test}


\bibliographystyle{siam}
\bibliography{Ref}

\end{document}